\documentclass[12pt]{article}
\usepackage[latin1]{inputenc}
\usepackage[british]{babel}
\usepackage{lmodern}
\usepackage[T1]{fontenc}
\usepackage[paper=a4paper, left=28mm, right=25mm, top=29mm, bottom=25mm]{geometry}
\usepackage{latexsym,amsfonts,amsmath,graphics}
\usepackage{epsfig}
\usepackage{enumitem}
\usepackage{amssymb,mathrsfs}
\usepackage{verbatim}

\newtheorem{theorem}{Theorem}
\newtheorem{lemma}{Lemma}
\newtheorem{corollary}{Corollary}
\newtheorem{proposition}{Proposition}

\newtheorem{definition}{Definition}

\newtheorem{remark}{Remark}
\newenvironment{proof}{\begin{trivlist} \item[\hskip\labelsep{\it Proof.}]}{$\hfill\Box$\end{trivlist}}
\newenvironment{proof2}{\begin{trivlist} \item[\hskip\labelsep{\it Proof of Theorem~\ref{theovdc}.}]}{$\hfill\Box$\end{trivlist}}
\newenvironment{proof3}{\begin{trivlist} \item[\hskip\labelsep{\it Proof of Theorem~\ref{theohamm}.}]}{$\hfill\Box$\end{trivlist}}
\newenvironment{proof4}{\begin{trivlist} \item[\hskip\labelsep{\it Proof of Theorem~\ref{theosym}.}]}{$\hfill\Box$\end{trivlist}}

\newcommand{\rd}{\,\mathrm{d}}
\newcommand{\bsd}{\boldsymbol{d}}
\newcommand{\bsj}{\boldsymbol{j}}
\newcommand{\bsx}{\boldsymbol{x}}

\newcommand{\bsz}{\boldsymbol{z}}
\newcommand{\bsl}{\boldsymbol{\ell}}
\newcommand{\bsk}{\boldsymbol{k}}
\newcommand{\bsm}{\boldsymbol{m}}
\newcommand{\bst}{\boldsymbol{t}}

\newcommand{\bszero}{\boldsymbol{0}}
\newcommand{\bsone}{\boldsymbol{1}}

\newcommand{\RR}{\mathbb{R}}

\newcommand{\NN}{\mathbb{N}}

\newcommand{\DD}{\mathbb{D}}
\newcommand{\BB}{\mathbb{B}}
\newcommand{\ii}{\mathrm{i}}

\newcommand{\sym}{{\rm sym}}
\newcommand{\cV}{\mathcal{V}}
\newcommand{\cR}{\mathcal{R}}
\newcommand{\cP}{\mathcal{P}}
\newcommand{\cD}{\mathcal{D}}
\newcommand{\cS}{\mathcal{S}}
\newcommand{\cF}{\mathcal{F}}

\newcommand{\eee}{{\rm e}^{\frac{2\pi\mathrm{i}}{b}\ell}-1}

\allowdisplaybreaks

\title{$L_p$- and $S_{p,q}^rB$-discrepancy of the \\ symmetrized van der Corput sequence and \\ modified Hammersley point sets in arbitrary bases}
\author{Ralph Kritzinger \thanks{The author is supported by the Austrian Science Fund (FWF): Project F5509-N26, which is a part of the Special Research Program "Quasi-Monte Carlo Methods: Theory and Applications".}}
\date{}

\begin{document}

\maketitle

\begin{abstract}
We study the local discrepancy of a symmetrized version of the well-known van der Corput sequence and of modified
two-dimensional Hammersley point sets in arbitrary base $b$. We give upper bounds on the norm of the local discrepancy in Besov spaces of dominating mixed smoothness $S_{p,q}^rB([0,1)^s)$, which will also give us bounds on the $L_p$-discrepancy. Our sequence and point sets will achieve the known optimal order for the $L_p$- and $S_{p,q}^rB$-discrepancy. The results in this paper generalize several previous results on $L_p$- and  $S_{p,q}^rB$-discrepancy estimates and provide a sharp upper bound on the $S_{p,q}^rB$-discrepancy of one-dimensional sequences for $r>0$. We will use the $b$-adic Haar function system in the proofs.
\end{abstract} 

\centerline{\begin{minipage}[hc]{130mm}{
{\em Keywords:} discrepancy, Besov spaces, van der Corput sequence, Hammersley point set\\
{\em MSC 2000:} 11K06, 11K31, 11K38, 42C10}
\end{minipage}}

 \allowdisplaybreaks
\section{Introduction and Statement of the Results} 

For an $N$-element point set $\cP=\{\bsx_0,\bsx_1,\dots,\bsx_{N-1}\}$ in the $s$-dimensional unit interval $[0,1)^s$ the local discrepancy $D_N(\cP,\bst)$ is defined as
$$ D_N(\cP,\bst):=\frac{1}{N}\sum_{n=0}^{N-1}\bsone_{[\bszero,\bst)}(\bsx_n)-\prod_{i=1}^{s}t_i. $$
In this expression, for $\bst=(t_1,\dots,t_s)\in [0,1]^s$, the notation $[\bszero,\bst)$ means the $s$-dimensional
interval $[0,t_1)\times \dots \times [0,t_s)$ with volume $\prod_{i=1}^{s}t_i$ and $\bsone_{I}$ denotes the indicator function of the interval $I\subseteq [0,1]^s$.
For an infinite sequence $\cS =(\bsx_n)_{n \ge 0}$ of elements in $[0,1)^s$
the local discrepancy $D_N(\cS,\bst)$ is defined as the local discrepancy of its first $N$ elements. \\
We denote the norm of the local discrepancy in a normed space $X$ of functions on $[0,1)^s$ by $\left\|D_N(\cP,\cdot)\mid X\right\|$, where we must require $D_N(\cP,\cdot) \in X$. \\
In this paper we are interested in particular normed spaces, namely the $L_p([0,1)^s)$ spaces and the Besov spaces $S_{p,q}^rB([0,1)^s)$ of dominating mixed smoothness. The definition of the latter is given in Section~\ref{Besov}. For $p\in [1,\infty]$, the $L_p([0,1)^s)$ space is defined as the collection of all functions $f$ on $[0,1)^s$ with finite $L_p([0,1)^s)$ norm, which for $1\leq p <\infty$ is defined as
$$ \left\|f\mid L_p([0,1)^s)\right\|:=\left(\int_{[0,1)^s}|f(\bst)|^p\mathrm{d} \, \bst\right)^{\frac{1}{p}}, $$
and for $p=\infty$ is given by $$\left\|f\mid L_{\infty}([0,1)^s)\right\|:=\sup_{\bst \in [0,1]^s}|f(\bst)|.$$
We speak of $\left\|D_N(\cP,\cdot)\mid L_p([0,1)^s)\right\|$ and $\left\|D_N(\cP,\cdot)\mid S_{p,q}^rB([0,1)^s))\right\|$ as the $L_p$- and the  $S_{p,q}^rB$-discrepancy of a point set $\cP \in [0,1)^s$, respectively. An analogous notation is used for sequences $\cS \in [0,1)^s$.  The $L_{\infty}$-discrepancy is the well-studied star discrepancy, but in this paper we will assume that $p \in [1,\infty)$.

The $L_p$-discrepancy is a quantitative measure for the irregularity of distribution of a sequence modulo one, see e.g. \cite{DT97,kuinie,mat}. It is also related to the worst-case integration error of a quasi-Monte Carlo rule, see e.g. \cite{DP10,LP14,Nied92}. The $S_{p,q}^rB$-discrepancy is related to the errors of quasi-Monte Carlo algorithms for numerical integration on spaces of dominating mixed smoothness, see e.g. \cite{Tri10}.\\

It is well known that for every $p\in(1,\infty)$ and for all $s\in \NN$ there exist positive numbers $c_{p,s}$ and $c'_{p,s}$ with the property that for every $N\geq 2$ any $N$-element point set $\cP$ in $[0,1)^s$ satisfies 
\begin{equation}\label{lowproinov}
\left\|D_N(\cP,\cdot)\mid L_p([0,1)^s)\right\| \ge c_{p,s} \frac{(\log N)^{\frac{s-1}{2}}}{N},
\end{equation}
and for every sequence $\cS$ in $[0,1)^s$ we have
\begin{equation}\label{lowproinovseq}
\left\|D_N(\cS,\cdot)\mid L_p([0,1)^s)\right\| \ge c'_{p,s} \frac{(\log N)^{\frac{s}{2}}}{N} \ \ \ \mbox{ for infinitely many $N \in \NN$},
\end{equation}
where $\log$ denotes the natural logarithm. The inequality \eqref{lowproinov} was shown by Roth~\cite{Roth} for $p=2$ (and therefore for $p \in (2,\infty)$ because of the monotonicity of the $L_p$ norms) and Schmidt~\cite{schX}  for $p\in (1,2)$.  Proinov~\cite{pro86} could prove \eqref{lowproinovseq} based on the results of Roth and Schmidt. Hal\'{a}sz~\cite{hala} showed that the bounds \eqref{lowproinov} and \eqref{lowproinovseq} also hold for the $L_1$-discrepancy of two-dimensional point sets and one-dimensional sequences, respectively. There exist point sets in every dimension $s$ with the order of the $L_p$-discrepancy of $(\log{N})^{\frac{s-1}{2}}/N$ for $p\in (1,\infty)$ (see \cite{chen1980} for the first existence result), which shows that the lower bound given in \eqref{lowproinov} is sharp. Chen and Skriganov \cite{CS02} gave for the first time for every integer $N \ge 2$ and every dimension $s \in \NN$, explicit constructions of finite $N$-element point sets in $[0,1)^s$ whose $L_2$-discrepancy achieves an order of convergence of $(\log{N})^{\frac{s-1}{2}}/N$. The result in \cite{CS02} was extended to the $L_p$-discrepancy for $p \in (1,\infty)$ by Skriganov~\cite{Skr}. The inequality \eqref{lowproinovseq} is also sharp for one-dimensional sequences (see e.g. \cite{KP2015}). Moreover, it is sharp for the $L_2$-discrepancy in all dimensions (see \cite{DP14neu, DP14}). Showing sharpness for all $p\in (1,\infty)$ in all dimensions is currently work in progress. \\

There are also known lower and upper bounds for the $S_{p,q}^rB$-discrepancy in arbitrary dimensions.
Triebel, who initiated the study of the local discrepancy in other spaces such as the Besov spaces and Triebel-Lizorkin spaces of dominating mixed smoothness in \cite{Tri10} and \cite{Tri10b}, showed that for all $1\leq p,q \leq \infty$ and $r\in \RR$ satisfying $\frac{1}{p}-1<r<\frac{1}{p}$ and $q<\infty$ if $p=1$ and $q>1$ if $p=\infty$ there exists a constant $c_1>0$ such that for any $N\geq 2$ the local discrepancy of any $N$-element point set $\cP$ in $[0,1)^s$ satisfies
\begin{equation} \label{lowerbound} \left\| D_N(\cP,\cdot) \mid S_{p,q}^rB([0,1)^s) \right\|\geq c_1 N^{r-1}(\log{N})^{\frac{s-1}{q}}. \end{equation}
Also, for any $N\geq 2$, there exists a point set $\cP$ in $[0,1)^s$ with $N$ points and a constant $c_2>0$ such that
$$ \left\| D_N(\cP,\cdot) \mid S_{p,q}^rB([0,1)^s) \right\|\leq c_2 N^{r-1}(\log{N})^{(s-1)\left(\frac{1}{q} +1-r\right)}. $$
Hinrichs showed in \cite{hin2010} that in two dimensions the gap between the exponents of the lower and the upper bounds can be closed for $1\leq p,q \leq \infty$ and $0 \leq r <\frac{1}{p}$. For the proof, he considered specific point sets, namely the dyadic versions of the digit scrambled Hammersley point sets given in Definition~\ref{def2}, which achieve a $S_{p,q}^rB$-discrepancy  of order in accordance to the lower bound \eqref{lowerbound}. Markhasin closed the gap in arbitrary dimensions under the same conditions on $p$, $q$ and $r$ by considering Chen-Skriganov
point sets in \cite{Mar2013a}. Summarizing, for $1\leq p,q \leq \infty$ and $r\geq 0$ there exist point sets $\cP$ in $[0,1)^s$ with $N$ points and a constant $c_3>0$ such that
\begin{equation*}  \left\| D_N(\cP,\cdot) \mid S_{p,q}^rB([0,1)^s) \right\|\leq c_3 N^{r-1}(\log{N})^{\frac{s-1}{q}}, \end{equation*}
which is best possible. Finding corresponding bounds on the $S_{p,q}^rB$-discrepancy for infinite sequences is work in progress. \\  

We introduce the $b$-adic van der Corput sequence and a symmetrized version thereof as well as the $b$-adic Hammersley point set  and two modified variants, namely a digit scrambled and a symmetrized version.
\begin{definition} \rm Let $b\geq 2$ be an integer and $\varphi_b(n)$ denote the {\it radical inverse} of $n \in \NN_0$ in base $b$. It is defined as $\varphi_b(n):=\sum_{i=0}^{k}n_ib^{-i-1}$ whenever $n$ has $b$-adic expansion $n=\sum_{i=0}^{k}n_ib^i$. The (classical) van der Corput sequence in base $b$ is the sequence $\cV_{b}=(\varphi_b(n))_{n\geq 0}$.
The symmetrized van der Corput sequence in base $b$ is given by $\cV_{b}^{\sym}:=(z_n)_{n\geq 0}$, where
\[z_n = \begin{cases} \varphi_b(m) &\mbox{if } n=2m, \\ 
                       1-\varphi_b(m) & \mbox{if } n=2m+1. \end{cases} 
\]
\end{definition}

\begin{definition} \rm \label{def2} Let $b\geq 2$ and $n \geq 1$ be integers. The (classical) $b$-adic Hammersley point set consisting of $N=b^n$ elements is given by
\begin{align*} \cR_{b,n}:=&\left\{\left(\frac{m}{b^n},\varphi_b(m)\right): m \in \{0,1,\dots,b^n-1\}\right\} \\ =&\left\{ \left(\frac{a_n}{b}+\dots+\frac{a_1}{b^n},\frac{a_1}{b}+\dots+\frac{a_n}{b^n}\right): a_1,\dots,a_n \in \{0,1,\dots,b-1\}\right\}. \end{align*}
Let $\mathfrak{S}_b$ be the set of all permutations of $\{0,1,\dots,b-1\}$ and let $\tau \in \mathfrak{S}_b$ be 
given by $\tau(k)=b-1-k$ for $k\in \{0,1,\dots,b-1\}$. Let $\Sigma=(\sigma_1,\dots,\sigma_{n}) \in \mathfrak{S}_b^n$. We define the digit scrambled Hammersley point set associated to $\Sigma$ consisting of $N=b^n$ elements by
$$ \cR_{b,n}^{\Sigma}:=\left\{ \left(\frac{\sigma_n(a_n)}{b}+\dots+\frac{\sigma_1(a_1)}{b^n},\frac{a_1}{b}+\dots+\frac{a_n}{b^n}\right): a_1,\dots,a_n \in \{0,1,\dots,b-1\}\right\}. $$
In this paper we assume that for a fixed $\sigma \in \mathfrak{S}_b$ we have either $\sigma_i=\sigma$ or $\sigma_i=\tau \circ \sigma =: \overline{\sigma}$ for all $i \in \{1,\dots,n\}$, i.e. $\Sigma \in \{\sigma,\overline{\sigma}\}^n$. We define the number \begin{equation} \label{ln} l_n:=|\{i\in \{1,\dots,n\}: \sigma_i=\sigma\}|, \end{equation}
i. e. the number of components $\sigma_i$ of $\Sigma$ which equal $\sigma$. \\
Let $\sigma \in \mathfrak{S}_b$ and $\Sigma=(\sigma_i)_{i=1}^n \in \{\sigma,\overline{\sigma}\}^n$ fixed. We also put $\Sigma^{\ast}=(\sigma_i^{\ast})_{i=1}^n \in \{\sigma,\overline{\sigma}\}^n$, where $\sigma_i^{\ast}=\tau\circ\sigma_i$ for all $i \in \{1,\dots,n\}$. The symmetrized Hammersley point set (associated to $\Sigma$) consisting of $\widetilde{N}=2b^n$ elements is then defined as 
   $$\cR_{b,n}^{\Sigma, \mathrm{sym}}=\cR_{b,n}^{\Sigma}\cup \cR_{b,n}^{\Sigma^{\ast}}.$$
We speak of a symmetrized point set, because $\cR_{b,n}^{\Sigma, \mathrm{sym}}$ can also be written as the union of $\cR_{b,n}^{\Sigma}$ with the point set
\begin{equation} \label{warumsym} \left\{ \left(1-\frac{1}{b^n}-x,y\right) \, : \, (x,y) \in \cR_{b,n}^{\Sigma} \right\} \end{equation}
\end{definition}

The process of symmetrization and digit scrambling of sequences and finite point sets has been applied in discrepancy theory many times before. This is due to the fact
that the classical versions of the van der Corput sequence and the Hammersley point set fail to have optimal $L_p$-discrepancy for all $p\in[1,\infty)$, which follows for instance from \cite[Theorem 1]{HKP14} and \cite[Remark 1]{KP2015}. The first two-dimensional point set with the optimal order of $L_2$-discrepancy was indeed found within symmetrized point sets by Davenport \cite{daven} in 1956. A thorough discussion of Davenport's principle, applied to the Hammersley point set, can be found in \cite{CS03}. Halton and Zaremba \cite{HZ} introduced digit scrambling for the dyadic Hammersley point set in 1969 and showed that the modified point sets overcome the defect of the classical Hammersley point set and achieve an optimal $L_2$-discrepancy in the sense of \eqref{lowproinov}. \\

The aim of this paper is to prove the following theorems. Here and throughout the paper, for functions $f,g:\NN \rightarrow \RR^+$, we write $g(N) \ll f(N)$ and $g(N) \gg f(N)$, 
if there exists a $C>0$ such that $g(N) \le C f(N)$ or $g(N) \geq C f(N)$ for all $N \in \NN$, $N \ge 2$, respectively. This constant $C$ is independent of $N$, but might depend on several other parameters. Further, we write $f(N) \asymp g(N)$ if the relations $g(N) \ll f(N)$ and $g(N) \gg f(N)$ hold simultaneously.  

\begin{theorem}\label{theovdc}
Let $1\leq p,q \leq \infty$ and $0\leq r < \frac{1}{p}$. Then for any integer $b\geq 2$ we have
$$ \left\|D_N(\cV_{b}^{\sym}) \mid S_{p,q}^rB([0,1))\right\| \ll N^{-1}(\log{N})^{\frac{1}{q}}  $$
if $r=0$ and
$$ \left\|D_N(\cV_{b}^{\sym}) \mid S_{p,q}^rB([0,1))\right\| \ll N^{r-1}  $$
if $0< r < \frac{1}{p}$ for all $N\geq 2$.
\end{theorem}

\begin{theorem} \label{theohamm} 
Let $1\leq p,q \leq \infty$ and $0\leq r < \frac{1}{p}$. Then for any integer $b\geq 2$ we have 
$$ \left\|D_N(\cR_{b,n}^{\Sigma}) \mid S_{p,q}^rB([0,1)^2)\right\| \ll N^{r-1}(\log{N})^{\frac{1}{q}} $$
if and only if $|2l_n-n|=O(n^{\frac{1}{q}})$ (where $l_n$ as defined in \eqref{ln}) or $\frac{1}{b}\sum_{a=0}^{b-1}\sigma(a)a=\frac{(b-1)^2}{4}$.
\end{theorem}

\begin{theorem} \label{theosym}  
Let $1\leq p,q \leq \infty$ and $0\leq r < \frac{1}{p}$. Then for any integer $b\geq 2$ we have 
$$ \left\|D_{\widetilde{N}}(\cR_{b,n}^{\Sigma,\sym}) \mid S_{p,q}^rB([0,1)^2)\right\| \ll \widetilde{N}^{r-1}(\log{\widetilde{N}})^{\frac{1}{q}} $$
independently of $\Sigma$.
\end{theorem}

\begin{remark} \rm In the theorems above we have to require $r<\frac{1}{p}$ to ensure that the indicator functions appearing in the definition of the local discrepancy are contained in $S_{p,q}^rB([0,1)^s)$ (see \cite[Proposition 6.3]{Tri10}). We must also require $r\geq 0$, since the symmetrized van der Corput sequence and the modified Hammersley point sets cannot provide optimal $S_{p,q}^rB$-discrepancy in the case $r<0$ as we will see in the proofs in Section~\ref{secproofthm}.
\end{remark}

To derive results on the $L_p$-discrepancy from the above theorems we use embedding theorems between the Besov space $S_{p,q}^rB([0,1)^s)$ and the
Triebel-Lizorkin space $S_{p,q}^rF([0,1)^s)$ of dominating mixed smoothness. Since we do not prove any results on the Triebel-Lizorkin space norm of the local discrepancy, we refer the interested reader to \cite{Mar2013, Mar2013a, Mar2013b, Tri10} for the definition of this space. In \cite{Tri10} we also find the embeddings $$S_{p,\min\{p,q\}}^rB([0,1)^s)\hookrightarrow S_{p,q}^rF([0,1)^s) \hookrightarrow S_{p,\max\{p,q\}}^rB([0,1)^s)$$ for $0<p,q \leq \infty$ and $$S_{p_1,q}^rF([0,1)^s) \hookrightarrow S_{q,q}^rB([0,1)^s) \hookrightarrow S_{p_2,q}^rF([0,1)^s)$$ for $0<p_2 \leq q\leq p_1<\infty$, which lead to the following corollary together with the identity $$S_{p,q}^rF([0,1)^s)=L_p([0,1)^s)$$ for $q=2$ and $r=0$ (see e.g. \cite[Remark 4.23]{Mar2013b}).

\begin{corollary}\label{lp} We have the following estimates of the $L_p$-discrepancy for $p \in [1,\infty)$ and all $b \geq 2$:
\begin{itemize}
  \item $\left\|D_N(\cV_{b}^{\sym}) \mid L_p([0,1))\right\| \ll N^{-1}(\log{N})^{\frac{1}{2}}$ for all $N\geq 2$,
	\item $\left\|D_N(\cR_{b,n}^{\Sigma}) \mid L_p([0,1)^2)\right\| \ll N^{-1}(\log{N})^{\frac{1}{2}},$ if and only if $|2l_n-n|=O(\sqrt{n})$ or $\frac{1}{b}\sum_{a=0}^{b-1}\sigma(a)a=\frac{(b-1)^2}{4}$
	\item $\left\|D_{\widetilde{N}}(\cR_{b,n}^{\Sigma,\sym}) \mid L_p([0,1)^2)\right\| \ll \widetilde{N}^{-1}(\log{\widetilde{N}})^{\frac{1}{2}}$ independently of $\Sigma$.
\end{itemize}
These inequalities show that we achieve optimal $L_p$-discrepancy with respect to the order of magnitude in $N$ or $\widetilde{N}$ in all three cases.
\end{corollary}

The structure of this paper is as follows: In the next Section~\ref{discussion}, we will show in which sense our results generalize previous results. In Section~\ref{Haarfkt}, we introduce the $b$-adic Haar function system and the Haar coefficients which will be a basic tool for our proofs. In Section~\ref{Besov}, we explain the Besov space norm and present a useful equivalent norm. Section~\ref{haarcoeff} is the most technical part of this paper and aims at finding upper bounds on the Haar coefficients of our sequences and point sets of interest. We will use these upper bounds in the subsequent Section~\ref{secproofthm} to finally prove the central theorems of this paper. 

\section{Discussion of the results}  \label{discussion}

To put our results into context, we point out in which sense they generalize previous results. Further, we provide a surprising insight into the optimal order of the $S_{p,q}^rB$-discrepancy of one-dimensional sequences for $r>0$. \\

It was recently proven in \cite{KP2015} that the symmetrized van der Corput sequence in base 2 achieves optimal
$L_p$-discrepancy for all $p \in [1,\infty)$. Corollary~\ref{lp} shows that the same is true for every base $b \geq 2$. \\

The digit scrambled Hammersley point sets in the sense of Definition~\ref{def2} were initially introduced by Faure \cite{fau81}. The $L_2$-discrepancy of these point sets was calculated exactly in \cite[Theorem 2]{FPPS09}. It follows from this exact formula, that the $L_2$-discrepancy of the digit scrambled Hammersley point set is of optimal order
$$N^{-1}(\log{N})^{\frac{1}{2}}$$ if and only if $|2l_n-n|=O(\sqrt{n})$  or $\frac{1}{b}\sum_{a=0}^{b-1}\sigma(a)a=\frac{(b-1)^2}{4}$ (with the notation in Definition~\ref{def2}). We remark that Corollary~\ref{lp} generalizes this fact to arbitrary $p\in [1,\infty)$. \\

Theorem~\ref{theohamm} can be regarded as a generalization of \cite[Theorem 1.1]{Mar2013}, where only digit scrambled Hammersley point sets with $\Sigma \in \{\mathrm{id}, \tau\}^n$ ($\mathrm{id}$ means the identity) were considered. By allowing general permutations $\sigma \in \mathfrak{S}_b$, it might happen that 
\begin{equation}\label{sigmagut} \frac{1}{b}\sum_{a=0}^{b-1}\sigma(a)a=\frac{(b-1)^2}{4}.\end{equation}
We therefore find a significantly higher number of two-dimensional point sets with optimal $S_{p,q}^rB$-discrepancy.
 We give examples for permutations $\sigma$ fulfilling \eqref{sigmagut} that were discovered in
\cite{FP09}. We choose $\sigma={\rm id}_l$ for $l\in \{0,1,\dots,b-1\}$, where
${\rm id}_l(a):=a \oplus l$ for $a\in \{0,1,\dots,b-1\} $ ($\oplus$ denotes addition modulo $b$). Then we have
\begin{align*}
   \sum_{a=0}^{b-1}{\rm id}_l(a)a&=  \sum_{a=0}^{b-1}(a \oplus l)a = \sum_{a=0}^{b-l-1}(a+l)a+\sum_{a=b-l}^{b-1}(a+l-b)a\\
	 &= \sum_{a=0}^{b-1}(a+l)a-b\sum_{a=b-l}^{b-1}a=\frac{b}{6}(1+2b^2+3l^2-3b(1+l)).
\end{align*}
Hence, \eqref{sigmagut} is fulfilled if and only if
$$ \frac{b^2-1}{12}=\frac{l(b-l)}{2}. $$
The pairs $(b,l)$ for which this equality is satisfied were given in \cite[Corollary 1]{FP09}. One could for instance choose $b=5$ and $l=1$. 
In \cite{FPPS09}, further explicit examples and constructions for permutations which fulfil \eqref{sigmagut} were presented. \\

It was shown in \cite{HKP14} that the symmetrized Hammersley point set in base 2 achieves optimal $L_p$-discrepancy for all $p\in[1,\infty)$. For the $L_2$-discrepancy, this result was already obtained in \cite{CS03} and \cite{lp2001} with aid of Walsh functions for slightly different variants of symmetrized Hammersley point sets. In the current paper, we give an appropriate definition of such symmetrized Hammersley point sets in arbitrary bases which have optimal $L_p$-discrepancy too (see Corollary~\ref{lp}). \\

Finally, we should comment on the results in Theorem~\ref{theovdc}. We notice that the logarithmic term $(\log{N})^{\frac{1}{q}}$ does only appear if $r=0$, whereas for $0<r<\frac{1}{p}$ we solely have the main term $N^{r-1}$. Thus, in the latter case we have the same optimal order of $S_{p,q}^rB$-discrepancy for one-dimensional sequences as for one-dimensional point sets, which is not the case for the $L_p$- or the star discrepancy. 
 The question arises whether the logarithmic term for $r=0$ can be omitted or not. This is certainly not the case for $q=2$, since what we get then is the $L_p$-discrepancy for which we have the inequality \eqref{lowproinovseq}, but may be the case for other values of $q$. Also, it would be very interesting to investigate if the fact that point sets and sequences achieve the same best possible order for the $S_{p,q}^rB$-discrepancy in the case $r>0$ appears in higher dimensions $s\geq 2$ too.

\section{The $b$-adic Haar basis} \label{Haarfkt}

In order to estimate the $L_p$- and the $S_{p,q}^rB$-discrepancy of $\cV_{b}^{\sym}$, $\cR_{b,n}^{\Sigma}$ and $\cR_{b,n}^{\Sigma,\sym}$ we use the Haar function system in base $b$. Haar functions are a useful and often applied tool in discrepancy theory, see e.g. \cite{DHP,hin2010,HKP14,KP2015,Mar2013,Mar2013a,Mar2013b}.
Additionally, Haar functions open the door for the investigation of the local discrepancy in further function spaces such as the BMO or the exponential Orlicz spaces (see e.g. \cite{bilyk}, another paper where dyadic digit scrambled Hammersley point sets were considered). \\

Let $b\geq 2$ be an integer. For $j \in \NN_0$ we define $\DD_j := \{0,1,\dots,b^j-1\}$ and $\BB_j:=\{1,\dots,b-1\}$. Additionally, we define the sets $\DD_{-1}:=\{0\}$ and $\BB_{-1}:=\{1\}$. For $j\in\NN_{0}$ and $m \in \DD_j$ we call the interval
$$ I_{j,m}:= \left[\frac{m}{b^j},\frac{m+1}{b^j}\right) $$
the $m$-th $b$-adic interval on level $j$. We also define $I_{-1,0}=\left[0,1\right)$, which is a
$b$-adic interval on level $0$. For $j \in \NN_{0}$, $m \in \DD_j$ and any $k\in \{0,1,\dots,b-1\}$ we introduce the interval 
$$I_{j,m}^k:=I_{j+1,bm+k}=\left[\frac{m}{b^j}+\frac{k}{b^{j+1}},\frac{m}{b^j}+\frac{k+1}{b^{j+1}}\right).$$ It is easy to see that $I_{j,m}=\bigcup_{k=0}^{b-1} I_{j,m}^k$ and $I_{j,m}^{k_1}\cap I_{j,m}^{k_2}=\emptyset$ whenever $k_1 \neq k_2$. We also put $I_{-1,0}^1=I_{-1,0}=\left[0,1\right)$. \\
For $j\in \NN_0$, $m \in \DD_j$ and $\ell \in \BB_j$ let $h_{j,m,\ell}$ be a function
on $\left[0,1\right)$ with support in $I_{j,m}$ and the constant value ${\rm e}^{\frac{2\pi\ii}{b}k\ell}$ on $I_{j,m}^k$ for any $k\in \{0,1,\dots, b-1\}$ and $0$ outside of
$I_{j,m}$. We call $h_{j,m,\ell}$ a $b$-adic Haar function on $\left[0,1\right)$. We also
put $h_{-1,0,1}=\bsone_{I_{-1,0}}=\bsone_{\left[0,1\right)}$ on $\left[0,1\right)$. It was shown in \cite[Theorem 2.1]{Mar2013} that the system
$$ \left\{ b^{\frac{\max\{0,j\}}{2}}h_{j,m,\ell}: j\in \NN_{-1}, m\in \DD_j, \ell \in \BB_j \right\}  $$
(where here and later on we use the abbreviation $\NN_{-1}:=\NN_{0}\cup\{-1\}$) is an orthonormal basis of $L_2(\left[0,1\right))$ and an unconditional basis of $L_p(\left[0,1\right))$ for $p \in (1,\infty)$. We speak of an one-dimensional $b$-adic Haar basis. \\

 To extend this definition to arbitrary dimensions $s$, for $\bsj=(j_1,j_2,\dots,j_s) \in \NN_{-1}^s$ and $\bsm=(m_1,m_2,\dots,m_s)\in \DD_{j_1}\times\DD_{j_2}\times \dots \times \DD_{j_s} =:\DD_{\bsj}$ we define the
$s$-dimensional $b$-adic interval $I_{\bsj,\bsm}:=I_{j_1,m_1}\times I_{j_2,m_2} \times \dots \times I_{j_s,m_s}$. For $\bsk=(k_1,k_2,\dots,k_s)$, where $k_i\in \{0,1,\dots,b-1\}$ if $j_i \in \NN_0$ and $k_i=1$ if $j_i=-1$ for $i\in\{1,\dots,s\}$, we put $I_{\bsj,\bsm}^{\bsk}:=I_{j_1,m_1}^{k_1}\times I_{j_2,m_2}^{k_2} \times \dots \times I_{j_s,m_s}^{k_s}$. Finally, for $\bsj=(j_1,j_2,\dots,j_s) \in \NN_{-1}^s$, $\bsm=(m_1,m_2,\dots,m_s)\in\DD_{\bsj}$ and $\bsl=(\ell_1,\ell_2,\dots,\ell_s)\in \BB_{j_1}\times \BB_{j_2} \times \dots \times \BB_{j_s}=:\BB_{\bsj}$ we define the $s$-dimensional $b$-adic Haar function $h_{\bsj,\bsm,\bsl}(\bsx):=h_{j_1,m_1,\ell_1}(x_1)h_{j_2,m_2,\ell_2}(x_2)\dots h_{j_s,m_s,\ell_s}(x_s)$ for $\bsx=(x_1,,x_2,\dots,x_s)\in[0,1)^s$. Then the system 
$$\left\{b^{\frac{|\bsj|}{2}}h_{\bsj,\bsm,\bsl}:\bsj \in \NN_{-1}^s,\bsm \in \DD_{\bsj}, \bsl \in \BB_{\bsj}\right\},$$ where $|\bsj|:=\max\{0,j_1\}+\max\{0,j_2\}+\dots +\max\{0,j_s\}$, is an orthonormal basis of $L_2(\left[0,1\right)^s)$ and an unconditional basis of $L_p(\left[0,1\right)^s)$ for $1<p<\infty$ (see again \cite[Theorem 2.1]{Mar2013}).
\\

The {\it Haar coefficients} of a function $f$ are defined as \begin{equation} \label{haarc}\mu_{\bsj,\bsm,\bsl}(f):=\langle f,h_{\bsj,\bsm,\bsl} \rangle=\int_{[0,1)^s} f(\bst)h_{\bsj,\bsm,\bsl}(\bst)\rd \bst\ \ \ \mbox{for $\bsj\in \NN_{-1}^s$, $\bsm \in \mathbb{D}_{\bsj}$} \mbox{ and } \bsl \in \BB_{\bsj}.\end{equation}

\section{The Besov spaces of dominating mixed smoothness} \label{Besov}

We give a definition of the classical dyadic Besov spaces of dominating mixed smoothness.
Let therefore $\cS(\RR^s)$ denote the Schwartz space and $\cS'(\RR^s)$ the space of tempered distributions on $\RR^s$.
For $f \in \cS'(\RR^s)$ we denote by $\cF f$ the Fourier transform of $f$ and by $\cF^{-1} f$ its inverse. Let $\phi_0\in \cS(\RR)$ satisfy $\phi_0(t)=1$ for $|t|\leq 1$ and $\phi_0(t)=0$ for $|t|> \frac{3}{2}$. Let
$$ \phi_d(t)=\phi_0(2^{-d}t)-\phi_0(2^{-d+1}t), $$
where $t\in \RR, d\in \NN$, and $\phi_{\bsd}(\bst)=\phi_{d_1}(t_1)\cdots \phi_{d_s}(t_s),$
where $\bsd=(d_1,\dots,d_s)\in\NN_0^s$, $\bst=(t_1,\dots,t_s)\in \RR^s$. We note that $\sum_{\bsd \in \NN_0^s}\phi_{\bsd}(\bst)=1$ for all $\bst \in \RR^s$. The functions $\cF^{-1}(\phi_{\bsd}\cF f)$ are entire analytic functions for any $f \in \cS'(\RR^s)$. Let $0<p,q\leq \infty$ and $r\in \RR$. The dyadic Besov space $S_{p,q}^rB(\RR^s)$ of dominating mixed smoothness consists of all $f \in \cS'(\RR^s)$ with finite quasi-norm
$$ \left\|f\mid S_{p,q}^rB(\RR^s) \right\| =\left(\sum_{\bsd \in \NN_0^s}2^{r(d_1+\dots+d_s)q}\left\|\cF^{-1}(\phi_{\bsd} \cF f) \mid L_p(\RR^s)\right\|^q\right)^{\frac{1}{q}}, $$
with the usual modification if $q=\infty$.
Let $\cD([0,1)^s)$ be the set of all complex-valued infinitely differentiable functions on $\RR^s$ with compact support in the interior of $[0,1)^s$ and let $\cD'([0,1)^s)$ be its dual space of all distributions in $[0,1)^s$. The Besov space $S_{p,q}^rB([0,1)^s)$ of dominating mixed smoothness on the domain $[0,1)^s$ consists of all functions $f \in \cD'([0,1)^s)$ with finite quasi norm
$$ \left\|f\mid S_{p,q}^rB([0,1)^s)\right\|=\inf{\left\{\left\|g \mid S_{p,q}^rB(\RR^s)\right\|:g \in S_{p,q}^rB(\RR^s), g|_{[0,1)^s}=f\right\}} $$
However, the dyadic definition of the Besov space norms is not suitable to estimate the discrepancy of sequences and point sets which are based on the $b$-adic expansion of integers. To overcome this drawback, $b$-adic versions of the Besov spaces $S_{p,q}^rB^b(\RR^s)$ and $S_{p,q}^rB^b([0,1)^s)$ have been introduced by Markhasin in $\cite{Mar2013, Mar2013b}$. We refer to these papers for the definition of the $b$-adic Besov spaces.
It was shown in \cite[Theorem 3.1]{Mar2013} that the $b$-adic Besov space $S_{p,q}^rB^b([0,1)^s)$ is equivalent to the classical dyadic
Besov space $S_{p,q}^rB([0,1)^s)$ and that we have the following useful characterization of functions which are contained in this space (see also \cite[Theorem 2.41]{Tri10} for the original proof of the dyadic case): \\
\begin{proposition} \label{equivalence} 
Let $0<p,q\leq \infty$ and $\frac{1}{p}-1<r<\min\left\{\frac{1}{p},1\right\}$. Let $f \in \cD'([0,1)^s)$.
Then $f \in S_{p,q}^rB^b([0,1)^s)$ if and only if it can be represented as
$$ f=\sum_{\bsj \in \NN_{-1}^s}\sum_{\bsm \in \DD_{\bsj},\bsl\in \BB_{\bsj}}\mu_{\bsj, \bsm, \bsl}b^{|\bsj|}h_{\bsj,\bsm,\bsl} $$
for some sequence $(\mu_{\bsj,\bsm,\bsl})$ satisfying
\[  \left(\sum_{\bsj \in \NN_{-1}^s}b^{(j_1+\dots+j_s)\left(r-\frac{1}{p}+1\right)q}\left( \sum_{\bsm \in \DD_{\bsj},\bsl\in \BB_{\bsj}}\left| \mu_{\bsj,\bsm,\bsl}\right|^p\right)^{\frac{q}{p}}\right)^{\frac{1}{q}}<\infty, \]
where the convergence is unconditional in $\cD'([0,1)^s)$ and in any $S_{pq}^{\rho}B^b([0,1)^s)$ with $\rho< r$. This representation of $f$ is unique with the $b$-adic Haar coefficients as defined in \eqref{haarc}.
The expression on the left-hand-side of the above inequality provides an equivalent quasi-norm on $S_{p,q}^rB^b([0,1)^s)$, i.e.
\begin{equation*} \left\|f\mid S_{p,q}^rB^b([0,1)^s) \right\| \asymp  \left(\sum_{\bsj \in \NN_{-1}^s}b^{(j_1+\dots+j_s)\left(r-\frac{1}{p}+1\right)q}\left( \sum_{\bsm \in \DD_{\bsj},\bsl\in \BB_{\bsj}}\left| \mu_{\bsj,\bsm,\bsl}\right|^p\right)^{\frac{q}{p}}\right)^{\frac{1}{q}}. \end{equation*}
\end{proposition}

\section{The Haar coefficients}  \label{haarcoeff}

\subsection{Haar coefficients of the symmetrized van der Corput sequence}

In the following, we will compute the Haar coefficients of the local discrepancy of $\cV_{b}^{\sym}$, i.e.
\[ \mu_{j,m,\ell}(D_N(\cV_{b}^{\sym},\cdot))=\langle D_N(\cV_{b}^{\sym},\cdot),h_{j,m,\ell} \rangle=\int_{0}^{1}D_N(\cV_{b}^{\sym},t)h_{j,m,\ell}(t)\rd t. \]
The proofs in this section are similar to those in \cite{KP2015}. 
Preceding the computation of the Haar coefficients, we collect some simple properties of the $b$-adic radical inverse function $\varphi_b(n)$ which
we will need in the proof of the essential Lemma~\ref{Theo1}. In the following, we will consequently omit the lower index $b$ in the radical inverse function,
since we will always consider an arbitrary but fixed base.

\begin{lemma} \label{Phi}
   The following relations hold for the radical inverse function $\varphi$ in base $b$:
   \begin{enumerate}
       \item  $\varphi(b^jw)=\frac{1}{b^j}\varphi(w)$ for all $j,w \in \NN_0$,
       \item \label{pt2} $\varphi(b^j\varphi(m))=\frac{m}{b^j}$ for all $j \in \NN_0$ and $m\in \{0,\dots,b^j-1\}$,
       \item \label{pt3} $\varphi(n) \in I_{j,m}$ if and only if $n=b^j\varphi(m)+b^jw$ for some $w \in \NN_0$, especially $\varphi(n) \in \stackrel{\circ}{I}_{j,m}$ if and only if $n=b^j\varphi(m)+b^jw$ for some $w \in \NN$,
			 \item $\varphi(n) \in I_{j,m}^k$ for some $k\in \{0,1,\dots,b-1\}$ if and only if $n=b^{j+1}\varphi(bm+k)+b^{j+1}w=b^j\varphi(m)+b^{j}k+b^{j+1}w$ for some $w \in \NN_0$, especially $\varphi(n) \in \mathring{I}_{j,m}^k$ if and only if $n=b^j\varphi(m)+b^{j}k+b^{j+1}w$ for some $w \in \NN$.
   \end{enumerate}
\end{lemma}

\begin{proof} 
The proofs of 1., 2. and 3. follow the same lines as \cite[Lemma 1]{KP2015}, whereas 4. is an immediate consequence of 3., regarding $I_{j,m}^k=I_{j+1,bm+k}$ and $\varphi(bm+k)=\frac{\varphi(m)}{b}+\frac{k}{b}$.
\end{proof}

The next lemma contains some formulas for exponential expressions which will occur in diverse parts of our proofs.

\begin{lemma} \label{exponential}
    The following equalities and inequalities hold for all integers $b\geq 2$:
		\begin{enumerate}
		   \item $\sum_{k=0}^{b-1}{\rm e}^{\frac{2\pi\ii}{b}k\ell}=0$ for all $\ell \in \{1,\dots,b-1\}$,
		   \item $\sum_{\ell=1}^{b-1}\frac{1}{|\eee|^2}= \frac{b^2-1}{12}$,
		 	\item $\sum_{\ell=1}^{b-1}\frac{1}{|\eee|^4}\leq \left(\frac{b^2-1}{12}\right)^2$,
			\item $\frac{1}{|\eee|}\leq \frac{2}{|\eee|^2}$ for all $\ell \in \{1,\dots,b-1\}$.
		\end{enumerate}
\end{lemma}

\begin{proof} The first item is a well-known result and can be verified by applying the formula for finite geometric
  sums. The proof of the second item can be found in \cite{DP10} or in \cite[Proposition 3.5]{Mar2013b}. The third item is an immediate consequence
  of this identity, since
   $$ \sum_{\ell=1}^{b-1}\frac{1}{|\eee|^4} \leq \left(\sum_{\ell=1}^{b-1}\frac{1}{|\eee|^2}\right)^2=\left(\frac{b^2-1}{12}\right)^2. $$
   The last item can be shown directly with aid of the triangle inequality:
	\[\frac{1}{|\eee|}=\frac{|\eee|}{|\eee|^2}
	     \leq \frac{|{\rm e}^{\frac{2\pi\ii}{b}\ell}|+|1|}{|\eee|^2}=\frac{2}{|\eee|^2}.\] 
\end{proof}

We start with the computation of the first Haar coefficient $\mu_{-1,0,1}$:

\begin{lemma} \label{erster}
   The Haar coefficient $\mu_{-1,0,1}$ of the local discrepancy $D_N(\cV_{b}^{\sym},\cdot)$ satisfies
   \[|\mu_{-1,0,1}|= \begin{cases} 0 &\mbox{if } N=2M, \\ 
                       \left|\frac{1}{2N}-\frac{\varphi(M)}{N}\right|\leq \frac{1}{2N}\ & \mbox{if } N=2M+1. \end{cases} 
\]
\end{lemma}

\begin{proof}
  The proof follows exactly the same lines as the proof of \cite[Lemma 2]{KP2015}.
\end{proof}

In the following, let $\mu_{j,m,\ell}^{N,\sym}$, $\mu_{j,m,\ell}^{N,\varphi}$ and $\mu_{j,m,\ell}^{N,1-\varphi}$ be the Haar coefficients of the local discrepancy of the first $N$ elements of the sequences $\cV_b^{\mathrm{sym}}$, $(\varphi(n))_{n\geq 0}$ and $(1-\varphi(n))_{n\geq 0}$, respectively. The next lemma may be proved in complete analogy to \cite[Corollary 1]{KP2015}.

\begin{lemma} \label{haarrelation}
 For all $j \in \NN_0$, $m \in \DD_j$ and $\ell \in \BB_j$ we have   
 \[ |\mu_{j,m,\ell}^{N,\sym}|\leq \left\{ 
\begin{array}{ll}
\frac{1}{2}\left(|\mu_{j,m,\ell}^{M,\varphi}|+|\mu_{j,m,\ell}^{M,1-\varphi}|\right) & \mbox{ if } N=2M,\\
\frac{1}{2M+1}\left((M+1)|\mu_{j,m,\ell}^{M+1,\varphi}|+M|\mu_{j,m,\ell}^{M,1-\varphi}|\right) & \mbox{ if } N=2M+1.
\end{array}\right.
 \]
\end{lemma}

We proceed with the calculation of the Haar coefficients of the local discrepancy in the case $j\in \NN_0$ and first prove the following general lemma.

\begin{lemma} \label{allgemein}
  Let $j \in \NN_0$, $m \in \mathbb{D}_j$ and $\ell \in \BB_j$. Then for the volume part $f(t)=t$ of the local discrepancy we have
  \[ \mu_{j,m,\ell}(f)=\frac{b^{-2j-1}}{\eee} \] and for the counting part $g(t)=\frac{1}{N}\sum_{n=0}^{N-1}\bsone_{\left[0,t\right)}(x_n)$
  we have
  \[ \mu_{j,m,\ell}(g)=\frac{b^{-j-1}}{N}\sum_{k=0}^{b-1}\sum_{\substack{n=0 \\ x_n \in {I}_{j,m}^k, \, x_n \neq \frac{m}{b^j}}}^{N-1}\left(\left(bm+k-b^{j+1}x_n\right){\rm e}^{\frac{2\pi\ii}{b}k\ell}-\sum_{r=0}^{k-1}{\rm e}^{\frac{2\pi\ii}{b}r\ell}\right), \]
  where the last sum is empty for $k=0$.
  \end{lemma}

\begin{proof}
   The assertion on $\mu_{j,m,\ell}(f)$ may be verified by simple integration. The Haar coefficients of $g$ are given by
	\begin{align*}\mu_{j,m,\ell}(g)=\int_{0}^{1}\left(\frac{1}{N}\sum_{n=0}^{N-1}\bsone_{\left[0,t\right)}(x_n)h_{j,m,\ell}(t)\right)\rd t = \frac{1}{N}\sum_{n=0}^{N-1}\underbrace{\int_{0}^{1}\bsone_{\left[0,t\right)}(x_n)h_{j,m,\ell}(t)\rd t}_{\mathcal{I}_n}.
   \end{align*}
	It is obvious that $\mathcal{I}_n=0$ in case that $x_n \notin I_{j,m}$ or $x_n=\frac{m}{b^j}$. Now we assume that $x_n\in I_{j,m}^k$ for some $k=0,1,\dots,b-1$ and $x_n\neq \frac{m}{b^j}$. Then we have
	\begin{align*}
	    \mathcal{I}_n&= \int_{x_n}^{\frac{m}{b_j}+\frac{k+1}{b^{j+1}}}{\rm e}^{\frac{2\pi\ii}{b}k\ell}\rd t+\sum_{r=k+1}^{b-1}\int_{I_{j,m}^r}{\rm e}^{\frac{2\pi\ii}{b}r\ell}\rd t \\
			&= b^{-j-1} \left(\left(bm+k+1-b^{j+1}x_n\right){\rm e}^{\frac{2\pi\ii}{b}k\ell}+\sum_{r=k+1}^{b-1}{\rm e}^{\frac{2\pi\ii}{b}r\ell}\right) \\
			&= b^{-j-1} \left(\left(bm+k+1-b^{j+1}x_n\right){\rm e}^{\frac{2\pi\ii}{b}k\ell}-\sum_{r=0}^{k}{\rm e}^{\frac{2\pi\ii}{b}r\ell}\right) \\
			&= b^{-j-1} \left(\left(bm+k-b^{j+1}x_n\right){\rm e}^{\frac{2\pi\ii}{b}k\ell}-\sum_{r=0}^{k-1}{\rm e}^{\frac{2\pi\ii}{b}r\ell}\right) 
	\end{align*}
	and the proof of this lemma is done.	
\end{proof}

Now we are ready to show a central lemma.

\begin{lemma} \label{Theo1}
   We have
   $$|\mu_{j,m,\ell}^{N,\varphi}|\leq\frac{1}{N}\frac{1}{b^j}\frac{9}{|\eee|^2}\ \ \ \mbox{ and } \ \ \ |\mu_{j,m,\ell}^{N,1-\varphi}|\leq\frac{1}{N}\frac{1}{b^j}\frac{15}{|\eee|^2}$$
   for all $0\le j < \lceil\log_{b}{N}\rceil$ and
   $$ |\mu_{j,m,\ell}^{N,\varphi}|=|\mu_{j,m,\ell}^{N,1-\varphi}|=\frac{b^{-2j-1}}{|\eee|} $$
   for all $j \geq \lceil\log_{b}{N}\rceil$.
\end{lemma}

\begin{proof} 
We start with $x_n=\varphi(n)$ and therefore employ Lemma~\ref{Phi}. It allows
us to display the sum
\[\sum_{k=0}^{b-1}\sum_{\substack{n=0 \\ \varphi(n) \in {I}_{j,m}^k}}^{N-1}\left(\left(bm+k-b^{j+1}\varphi(n)\right){\rm e}^{\frac{2\pi\ii}{b}k\ell}-\sum_{r=0}^{k-1}{\rm e}^{\frac{2\pi\ii}{b}r\ell}\right),\]
which appears in Lemma~\ref{allgemein}, as
\[\sum_{k=0}^{b-1}\sum_{w=0}^{A(k)}\left(\left(bm+k-b^{j+1}\varphi\left(b^j\varphi(m)+b^{j}k+b^{j+1}w\right)\right){\rm e}^{\frac{2\pi\ii}{b}k\ell}-\sum_{r=0}^{k-1}{\rm e}^{\frac{2\pi\ii}{b}r\ell}\right).\]
We may include the case $\varphi(n)=\frac{m}{b^j}$ since the corresponding summand is zero anyway.
In the above expression, $A(k):=\left\lfloor \frac{N-1}{b^{j+1}}-\frac{\varphi(m)}{b}-\frac{k}{b} \right\rfloor$. We choose this value for the upper index of the sum, since we have to
ensure that the conditions $0\le n \le N-1$ and $n=b^j\varphi(m)+b^{j}k+b^{j+1}w$ are fulfilled simultaneously.
With aid of Lemma~\ref{Phi}, 1. and 2., which gives
$$\varphi\left(b^j\varphi(m)+b^{j}k+b^{j+1}w\right)=\frac{m}{b^j}+\frac{1}{b^j}\varphi(k+bw)=\frac{m}{b^j}+\frac{1}{b^j}\left(\frac{k}{b}+\frac{1}{b}\varphi(w)\right),
$$
the above expression can be simplified to
\[-\sum_{k=0}^{b-1}\sum_{w=0}^{A(k)}\left(\varphi(w){\rm e}^{\frac{2\pi\ii}{b}k\ell}+\sum_{r=0}^{k-1}{\rm e}^{\frac{2\pi\ii}{b}r\ell}\right)=:S.\]
Next we notice that $A(k)$ can only take two possible values, namely $A(k)=\tilde{A}$ or
$A(k)=\tilde{A}-1$, where $\tilde{A}=\left\lfloor \frac{N-1}{b^{j+1}}-\frac{\varphi(m)}{b} \right\rfloor$. We assume that $k_0 \in \{1,\dots,b\}$ is such that $A(k)=\tilde{A}$ for
$k \in \{0,\dots,k_0-1\}$ and, in case that $k_0 < b$, $A(k)=\tilde{A}-1$ for $k \in \{k_0,\dots,b-1\}$. Hence, we can write
\begin{align*}
    S=& -\underbrace{\sum_{k=0}^{k_0-1}\left(\varphi(\tilde{A}){\rm e}^{\frac{2\pi\ii}{b}k\ell}+\sum_{r=0}^{k-1}{\rm e}^{\frac{2\pi\ii}{b}r\ell}\right)}_{S_1}-\underbrace{\sum_{k=0}^{b-1}\sum_{w=0}^{\tilde{A}-1}\left(\varphi(w){\rm e}^{\frac{2\pi\ii}{b}k\ell}+\sum_{r=0}^{k-1}{\rm e}^{\frac{2\pi\ii}{b}r\ell}\right)}_{S_2}.
\end{align*}
We intend to simplify $S_2$ and therefore change the order of the sums to obtain
\begin{align*}S_2&=\sum_{w=0}^{\tilde{A}-1}\varphi(w)\sum_{k=0}^{b-1}{\rm e}^{\frac{2\pi\ii}{b}k\ell}+\sum_{w=0}^{\tilde{A}-1}\sum_{k=0}^{b-1}\sum_{r=0}^{k-1}{\rm e}^{\frac{2\pi\ii}{b}r\ell}\\
   &=\frac{\tilde{A}}{\eee}\sum_{k=0}^{b-1}\left({\rm e}^{\frac{2\pi\ii}{b}k\ell}-1\right) 
	 =-\frac{b\tilde{A}}{\eee}.
\end{align*}
We combine the previous results with Lemma~\ref{allgemein} to obtain
\begin{align*}
    \mu_{j,m,\ell}^{N,\varphi}&=\frac{1}{\eee}\left(\frac{1}{N}\frac{1}{b^j}\tilde{A}-b^{-2j-1}\right)-\frac{1}{N}\frac{1}{b^{j+1}}S_1
\end{align*}
Now we take the absolute value and apply the triangle inequality. This yields
$$ |\mu_{j,m,\ell}^{N,\varphi}|\leq\frac{1}{|\eee|}\left|\frac{1}{N}\frac{1}{b^j}\tilde{A}-b^{-2j-1}\right|+\frac{1}{N}\frac{1}{b^{j+1}}|S_1|. $$
Since the inequalities $x-1< \lfloor x \rfloor \leq x$ for all $x\in\RR$ yield
$$ \frac{1}{N}\frac{1}{b^j}\tilde{A}-b^{-2j-1}\leq \frac{1}{N}\frac{1}{b^j}\left(\frac{N-1}{b^{j+1}}-\frac{\varphi(m)}{b}\right)-b^{-2j-1}=
   -\frac{1}{N}\frac{1}{b^{2j+1}}-\frac{1}{N}\frac{\varphi(m)}{b^{j+1}}<0  $$
and
\begin{align*} 
  \frac{1}{N}\frac{1}{b^j}\tilde{A}-b^{-2j-1}&\geq \frac{1}{N}\frac{1}{b^j}\left(\frac{N-1}{b^{j+1}}-\frac{\varphi(m)}{b}-1\right)-b^{-2j-1} \\
             &= -\frac{1}{N}\frac{1}{b^{2j+1}}-\frac{1}{N}\frac{\varphi(m)}{b^{j+1}}-\frac{1}{N}\frac{1}{b^j} \geq -\frac{1}{N}\frac{b+2}{b^{j+1}}\geq -\frac{1}{N}\frac{2}{b^j},
\end{align*}
we get $$ \left|\mu_{j,m,\ell}^{N,\varphi}\right|\leq \frac{1}{|\eee|}\frac{1}{N}\frac{2}{b^j}+\frac{1}{N}\frac{1}{b^{j+1}}|S_1|.$$
It remains to estimate $|S_1|$. We have
\begin{align*} 
  |S_1|&\leq \varphi(\tilde{A})\left|\sum_{k=0}^{k_0-1}{\rm e}^{\frac{2\pi\ii}{b}k\ell}\right|+\left|\sum_{k=0}^{k_0-1}\sum_{r=0}^{k-1}{\rm e}^{\frac{2\pi\ii}{b}r\ell}\right| \\
      &\leq \frac{|{\rm e}^{\frac{2\pi\ii}{b}k_0\ell}-1|}{|\eee|}+\left|\frac{1}{\eee}\left(\frac{{\rm e}^{\frac{2\pi\ii}{b}k_0\ell}-1}{\eee}-k_0\right)\right| \\
      &\leq \frac{2}{|\eee|}+\frac{2}{|\eee|^2}+\frac{b}{|\eee|} \leq \frac{5b}{|\eee|^2},
\end{align*}
where we used Lemma~\ref{exponential} in the last step. Altogether, we have verified
$$ \left|\mu_{j,m,\ell}^{N,\varphi}\right|\leq \frac{1}{|\eee|}\frac{1}{N}\frac{2}{b^j}+\frac{1}{|\eee|^2}\frac{1}{N}\frac{5}{b^{j}}
   \leq \frac{1}{N}\frac{1}{b^j}\frac{9}{|\eee|^2}.$$
 This proves the first estimate of the lemma. It follows from Lemma~\ref{Phi}, that there are no elements of $\{x_0,x_1,\dots,x_{N-1}\}$ contained in the interior of $I_{j,m}$, if $b^j\varphi(m)+b^j\geq N$,
 which is certainly fulfilled if $j \geq \left\lceil \log_{b}{N}\right\rceil$. Therefore, in this case the counting part does not contribute to the Haar coefficient $\mu_{j,m,\ell}^{N,\varphi}$ as we have seen in the proof of Lemma~\ref{allgemein}, and we have
 $$|\mu_{j,m,\ell}|=\frac{b^{-2j-1}}{|\eee|}$$
as claimed. \\

We investigate $\left|\mu_{j,m,\ell}^{N,1-\varphi}\right|$. Therefore we denote by
$\mathring{I}^k_{j,m}$ the interior of the interval $I_{j,m}^k$ and write
\begin{align*}
   \sum_{k=0}^{b-1}\sum_{\substack{n=0 \\ 1-\varphi(n) \in I_{j,m}^k\\ 1-\varphi(n)\neq\frac{m}{b^j}}}^{N-1}\left(\left(bm+k-b^{j+1}(1-\varphi(n))\right){\rm e}^{\frac{2\pi\ii}{b}k\ell}-\sum_{r=0}^{k-1}{\rm e}^{\frac{2\pi\ii}{b}r\ell}\right)
\end{align*}
as 
\begin{align*}
   &\sum_{k=0}^{b-1}\sum_{\substack{n=0 \\ 1-\varphi(n) \in \mathring{I}_{j,m}^k}}^{N-1}\left(\left(bm+k-b^{j+1}(1-\varphi(n))\right){\rm e}^{\frac{2\pi\ii}{b}k\ell}-\sum_{r=0}^{k-1}{\rm e}^{\frac{2\pi\ii}{b}r\ell}\right) \\
	&-\sum_{k=1}^{b-1}\sum_{\substack{n=0 \\ 1-\varphi(n)=\frac{m}{b^j}+\frac{k}{b^{j+1}}}}^{N-1}\sum_{r=0}^{k-1}{\rm e}^{\frac{2\pi\ii}{b}r\ell} =:T_1-T_2.
\end{align*}
It is easily shown that $1-\varphi(n) \in \mathring{I}_{j,m}^k$ if and only if $\varphi(n) \in \mathring{I}_{j,b^j-m-1}^{b-k-1}$. In analogy to preceding parts of this proof we define $B(k):=\left\lfloor \frac{N-1}{b^{j+1}}-\frac{\varphi(b^j-m-1)}{b}-\frac{b-k-1}{b} \right\rfloor$ as well as
$\tilde{B}:=\left\lfloor \frac{N-1}{b^{j+1}}-\frac{\varphi(b^j-m-1)}{b} \right\rfloor$.
Let $k_0'\in \{0,\dots,b-1\}$ be such that $B(k)=\tilde{B}$ for $k\in \{k_0',\dots,b-1\}$ and, in case that $k_0'> 0$, $B(k)=\tilde{B}-1$ for $k\in \{0,\dots,k_0'-1\}$.
We apply Lemma~\ref{Phi} to obtain
\begin{align*}
   T_1=&\sum_{k=0}^{b-1}\sum_{\substack{n=0 \\ \varphi(n) \in \mathring{I}_{j,b^j-m-1}^{b-k-1}}}^{N-1}\left(\left(bm+k-b^{j+1}(1-\varphi(n))\right){\rm e}^{\frac{2\pi\ii}{b}k\ell}-\sum_{r=0}^{k-1}{\rm e}^{\frac{2\pi\ii}{b}r\ell}\right)  \\
	=&\sum_{k=0}^{b-1}\sum_{w=1}^{B(k)}\bigg(\left(bm+k-b^{j+1}(1-\varphi(b^{j+1}\varphi(b(b^j-m-1)+b-k-1)+b^{j+1}w))\right){\rm e}^{\frac{2\pi\ii}{b}k\ell}  \\  &\hspace{1.5 cm}-\sum_{r=0}^{k-1}{\rm e}^{\frac{2\pi\ii}{b}r\ell}\bigg) \\
	=& -\sum_{k=0}^{b-1}\sum_{w=1}^{B(k)}\left((\varphi(w)+1){\rm e}^{\frac{2\pi\ii}{b}k\ell}+\sum_{r=0}^{k-1}{\rm e}^{\frac{2\pi\ii}{b}r\ell}\right) \\
	=& -\underbrace{\sum_{k=k_0'}^{b-1}\left((\varphi(\tilde{B})+1){\rm e}^{\frac{2\pi\ii}{b}k\ell}+\sum_{r=0}^{k-1}{\rm e}^{\frac{2\pi\ii}{b}r\ell}\right)}_{T_{1,1}}-\underbrace{\sum_{k=0}^{b-1}\sum_{w=1}^{\tilde{B}-1}\left((\varphi(w)+1){\rm e}^{\frac{2\pi\ii}{b}k\ell}+\sum_{r=0}^{k-1}{\rm e}^{\frac{2\pi\ii}{b}r\ell}\right)}_{T_{1,2}}.
\end{align*}
 Similarly as above, we can show that $T_{1,2}=-\frac{b(\tilde{B}-1)}{\eee}$. Altogether, we have
$$\mu_{j,m,\ell}^{N,1-\varphi}=\frac{1}{\eee}\left(\frac{b^{-j}}{N}(\tilde{B}-1)-b^{-2j-1}\right)-\frac{b^{-j-1}}{N}(T_{1,1}+T_{2}),$$
and so the triangle inequality gives
$$ |\mu_{j,m,\ell}^{N,1-\varphi}|\leq\frac{1}{|\eee|}\left|\frac{b^{-j}}{N}(\tilde{B}-1)-b^{-2j-1}\right|+\frac{b^{-j-1}}{N}(|T_{1,1}|+|T_{2}|). $$
One can check in the same manner as done above that $\left|\frac{b^{-j}}{N}(\tilde{B}-1)-b^{-2j-1}\right|\leq \frac{1}{N}\frac{3}{b^j}$ and $|T_{1,1}|\leq \frac{7b}{|\eee|^2}$. We also find
$$ |T_2|\leq \left|\sum_{k=1}^{b-1}\sum_{r=0}^{k-1}{\rm e}^{\frac{2\pi\ii}{b}r\ell}\right|\leq \frac{1}{|\eee|}
 \left|\sum_{k=0}^{b-1}\left({\rm e}^{\frac{2\pi\ii}{b}k\ell}-1\right)\right|=\frac{b}{|\eee|}\leq \frac{2b}{|\eee|^2}.$$
By combining all these results we finally arrive at
$$ |\mu_{j,m,\ell}^{N,1-\varphi}|\leq \frac{1}{|\eee|^2}\left(\frac{6}{N}\frac{1}{b^j}+\frac{7}{N}\frac{1}{b^j}+\frac{2}{N}\frac{1}{b^j}\right)=\frac{1}{N}\frac{1}{b^j}\frac{15}{|\eee|^2}. $$
The equality $|\mu_{j,m,\ell}^{N,1-\varphi}|= \frac{b^{-2j-1}}{|\eee|}$ for $j \geq \left\lceil \log_{b}{N} \right\rceil$ can be verified analogously as in the case of $|\mu_{j,m,\ell}^{N,\varphi}|$ and the proof of the lemma is complete.
\end{proof}

\begin{corollary} \label{coro1} The Haar coefficients of the symmetrized van der Corput sequence in base $b$ for $j \in \NN_0$ satisfy
$$|\mu_{j,m,\ell}^{N,\sym}|\left\{ 
\begin{array}{ll}
\leq \frac{1}{N}\frac{1}{b^j}\frac{26}{|\eee|^2} & \mbox{ if } j < \lceil \log_{b}{N}\rceil,\\ \\
= \frac{b^{-2j-1}}{|\eee|} & \mbox{ if } j\geq \lceil \log_{b}{N}\rceil.
\end{array}\right.$$
\end{corollary}
\begin{proof} We combine Lemma~\ref{haarrelation} and Lemma~\ref{Theo1} to obtain the result. 
\end{proof}

\subsection{Haar coefficients of the modified Hammersley point sets} \label{Haarcoeff}
In this subsection, we will compute the Haar coefficients of the local discrepancy of $\cR_{b,n}^{\Sigma}$, i.e.,
\[ \mu_{\bsj,\bsm,\bsl}(D_N(\cR_{b,n}^{\Sigma},\cdot))=\langle D_N(\cR_{b,n}^{\Sigma},\cdot),h_{\bsj,\bsm,\bsl} \rangle=\int_{[0,1]^2}D_N(\cR_{b,n}^{\Sigma},\bst)h_{\bsj,\bsm,\bsl}(\bst)\rd \bst. \]

We follow closely the ideas in \cite{Mar2013} and \cite{Mar2013b}. The main and essential difference lies in the calculation of the first
Haar coefficient $\mu_{(-1,-1),(0,0),(1,1)}$ of the local discrepancy $D_N(\cR_{b,n}^{\Sigma},\cdot)$, which we carry out in the following lemma.
\begin{lemma} \label{haar1} Let $n\in\NN$, $\sigma\in\mathfrak{S}_b$, $\Sigma=(\sigma_1,\dots,\sigma_n)\in\{\sigma,\overline{\sigma}\}^n$ and $l_n=|\{i\in\{1,\dots,n\}: \sigma_i=\sigma\}|$ as defined in \eqref{ln}. Then we have 
$$ \mu_{(-1,-1),(0,0),(1,1)}(D_N(\cR_{b,n}^{\Sigma},\cdot))=\frac{1}{N}(n-2l_n)\left(\frac{(b-1)^2}{4b}-\frac{1}{b^2}\sum_{a=0}^{b-1}\sigma(a)a\right)+\frac{1}{2N}+\frac{1}{4N^2}, $$
where $N=b^n$ is the number of elements in $\cR_{b,n}^{\Sigma}$.
\end{lemma}
\begin{proof}
   We denote the points of $\cR_{b,n}^{\Sigma}$ by $\{\bsx_0,\dots,\bsx_{N-1}\}$, where $\bsx_r=(x_r^{(1)},x_r^{(2)})$ for
	all $r \in \{0,1,\dots,N-1\}$. We have
	\begin{align}
	  \mu_{(-1,-1),(0,0),(1,1)}&=\int_0^1\int_0^1D_N(\cR_{b,n}^{\Sigma},(t_1,t_2))\rd t_1 \rd t_2 \nonumber \\ \nonumber
	  &=\frac{1}{N}\sum_{r=0}^{N-1}\int_0^1\int_0^1\bsone_{\left[0,t_1\right)\times \left[0,t_2\right)}(\bsx_r)\rd t_1 \rd t_2-\int_0^1\int_0^1t_1t_2 \rd t_1 \rd t_2 \\
	  &=\frac{1}{N}\sum_{r=0}^{N-1}(1-x_r^{(1)})(1-x_r^{(2)})-\frac{1}{4}\nonumber \\ \nonumber
		  &= \frac{3}{4}-\frac{1}{N}\sum_{r=0}^{N-1}x_r^{(1)}-\frac{1}{N}\sum_{r=0}^{N-1}x_r^{(2)}+\frac{1}{N}\sum_{r=0}^{N-1}x_r^{(1)}x_r^{(2)} \\ \nonumber
			&=\frac{3}{4}-\frac{2}{N}\sum_{r=0}^{N-1}\frac{r}{N}+\frac{1}{N}\sum_{r=0}^{N-1}x_r^{(1)}x_r^{(2)} \\ \nonumber
			&= \frac{3}{4}-\frac{2}{N}\left(\frac{N-1}{2}\right)+\frac{1}{N}\sum_{r=0}^{N-1}x_r^{(1)}x_r^{(2)} \\ 
			&= -\frac{1}{4}+\frac{1}{N}+\frac{1}{N}\sum_{r=0}^{N-1}x_r^{(1)}x_r^{(2)}. \label{eins}
	\end{align}
	We regarded the obvious fact that $\sum_{r=0}^{N-1}x_r^{(1)}=\sum_{r=0}^{N-1}x_r^{(2)}=\sum_{r=0}^{N-1}\frac{r}{N}$.
It remains to investigate the sum $S:=\sum_{r=0}^{N-1}x_r^{(1)}x_r^{(2)}$. We have
\begin{align*}
   S&= \sum_{a_1,\dots,a_n=0}^{b-1}\left(\frac{\sigma_n(a_n)}{b}+\dots+\frac{\sigma_1(a_1)}{b^n}\right)\left(\frac{a_1}{b}+\dots+\frac{a_n}{b^n}\right) 
	  = \sum_{a_1,\dots,a_n=0}^{b-1}\sum_{k_1,k_2=1}^{n}\frac{\sigma_{k_1}(a_{k_1})a_{k_2}}{b^{n+1-k_1}b^{k_2}}. \end{align*}
	  
Next we distinguish between the cases where $k_1=k_2=k$ and where $k_1\neq k_2$ and change the orders of the sums, which results in 
	 
		\[ S=\underbrace{\sum_{k=1}^{n}b^{n-1}\sum_{a_k=0}^{b-1}\frac{\sigma_k(a_{k})a_{k}}{b^{n+1}}}_{S_1}+\underbrace{\sum_{\substack{k_1,k_2=1 \\ k_1 \neq k_2}}^{n}b^{n-2}\sum_{a_{k_1},a_{k_2}=0}^{b-1}\frac{\sigma_{k_1}(a_{k_1})a_{k_2}}{b^{n+1-k_1}b^{k_2}}}_{S_2}. 
\]

The factors $b^{n-1}$ and $b^{n-2}$ come from the fact that the summands of $S_1$ and $S_2$ only depend on the digit $a_k$ or on the digits $a_{k_1}$ and $a_{k_2}$, respectively, and hence the sums over the remaining digits give a factor $b$ each. Now we have
\begin{align*}
   S_1&= \frac{1}{b^2}\sum_{\substack{k=1 \\ \sigma_k=\sigma}}^{n}\sum_{a_k=0}^{b-1}\sigma(a_k)a_k+\frac{1}{b^2}\sum_{\substack{k=1 \\ \sigma_k=\overline{\sigma}}}^{n}\sum_{a_k=0}^{b-1}\overline{\sigma}(a_k)a_k \\
      &=\frac{l_n}{b^2}\sum_{a=0}^{b-1}\sigma(a)a+\frac{(n-l_n)}{b^2}\sum_{a=0}^{b-1}\overline{\sigma}(a)a \\
      &=\frac{l_n}{b^2}\sum_{a=0}^{b-1}\sigma(a)a+\frac{(n-l_n)}{b^2}\sum_{a=0}^{b-1}(b-1-\sigma(a))a \\
      &=\frac{l_n-(n-l_n)}{b^2}\sum_{a=0}^{b-1}\sigma(a)a+\frac{n-l_n}{b^2}(b-1)\sum_{a=0}^{b-1}a \\
      &=\frac{2l_n-n}{b^2}\sum_{a=0}^{b-1}\sigma(a)a+\frac{n-l_n}{2b}(b-1)^2
\end{align*}
and
\begin{align*}
   S_2&= \frac{1}{b^3}\sum_{\substack{k_1,k_2=1 \\ k_1 \neq k_2}}^{n}b^{k_1-k_2}\left(\sum_{a_{k_1}=0}^{b-1}\sigma_{k_1}(a_{k_1})\right)\left(\sum_{a_{k_2}=0}^{b-1}a_{k_2}\right) \\
   &=\frac{1}{b^3}\left(\frac{b(b-1)}{2}\right)^2\sum_{\substack{k_1,k_2=1 \\ k_1 \neq k_2}}^{n}b^{k_1-k_2}. \end{align*} 
Straight-forward algebra yields
$$ \sum_{\substack{k_1,k_2=1 \\ k_1 \neq k_2}}^{n}b^{k_1-k_2}=\sum_{k_1,k_2=1}^{n}b^{k_1-k_2}-\sum_{k=1}^{n}1=\frac{b}{(b-1)^2}(b^n+b^{-n}-2)-n, $$
which leads to
$$ S_2=\frac{1}{4}\left(\frac{1}{N}+N-2+n\left(2-\frac{1}{b}-b\right)\right).$$

Altogether we find
\begin{align*}
   S&=\frac{2l_n-n}{b^2}\sum_{a=0}^{b-1}\sigma(a)a+\frac{n-l_n}{2b}(b-1)^2+\frac{1}{4}\left(\frac{1}{N}+N-2+n\left(2-\frac{1}{b}-b\right)\right) \\
    &=(n-2l_n)\left(\frac{(b-1)^2}{4b}-\frac{1}{b^2}\sum_{a=0}^{b-1}\sigma(a)a\right)+\frac{1}{4N}+\frac{N}{4}-\frac{1}{2}.
\end{align*}
Inserting this into \eqref{eins} yields the formula for $\mu_{(-1,-1),(0,0),(1,1)}$.
\end{proof}

\begin{remark} \rm \label{perm} We note that the Haar coefficient $\mu_{(-1,-1),(0,0),(1,1)}$ is of order $\frac{\log{N}}{N}$ in general. 
It can be reduced to the order $\frac{(\log{N})^{\frac{1}{q}}}{N}$ (for $1\leq q \leq \infty$) however by either choosing $l_n$ such that $|2l_n-n|=O(n^{\frac{1}{q}})$ or
by choosing the permutation $\sigma$ such that
$$ \frac{1}{b}\sum_{a=0}^{b-1}\sigma(a)a=\frac{(b-1)^2}{4}. $$
We remark that these are exactly the conditions which appear in Theorem~\ref{theohamm} to assure the optimal $S_{p,q}^rB$-discrepancy for the digit scrambled Hammersley point set. 
\end{remark}

\begin{remark} \rm Markhasin computed the Haar coefficients for the case $\sigma={\rm id}$ in \cite{Mar2013} and \cite{Mar2013b}. He showed that
$$ \mu_{(-1,-1),(0,0),(1,1)}=\frac{1}{4}b^{-2n}+\frac{1}{2}b^{-n}+(2l_n-n)\frac{b^2-1}{12b} b^{-n},$$
This result is a special case of Lemma~\ref{haar1}.
\end{remark}

Now we turn to the estimation of the remaining Haar coefficients $\mu_{\bsj,\bsm,\bsl}$, where $\bsj\neq (-1,-1)$.
For that purpose, we state a lemma that can also be found in \cite[Lemma 4.2, 4.3]{Mar2013} and is the two-dimensional
analogon of Lemma~\ref{allgemein}.
\begin{lemma}[Markhasin] \label{linearpart} Let $f(\bst)=t_1t_2$ for $\bst=(t_1,t_2)\in \left[0,1\right)^2$. Let $\bsj\in \NN_{-1}^2$, $\bsm\in \DD_{\bsj}$, $\bsl \in \BB_{\bsj}$ and let $\mu_{\bsj,\bsm,\bsl}(f)$ be the $b$-adic
Haar coefficient of $f$. Then
\begin{enumerate}
   \item If $\bsj=(j_1,j_2)\in \NN_{0}^2$, then $$ \mu_{\bsj,\bsm,\bsl}(f)=\frac{b^{-2j_1-2j_2-2}}{\left({\rm e}^{\frac{2\pi\ii}{b}\ell_1}-1\right)\left({\rm e}^{\frac{2\pi\ii}{b}\ell_2}-1\right)}. $$
	\item If $\bsj=(j_1,-1)$ or $\bsj=(-1,j_2)$ with $j_1 \in \NN_0$ or $j_2 \in \NN_0$, then
	  $$ \mu_{\bsj,\bsm,\bsl}(f)=\frac{1}{2}\frac{b^{-2j_i-1}}{{\rm e}^{\frac{2\pi\ii}{b}\ell_i}-1} \text{  with  } i=1 \text{  or  } i=2, \text{  respectively}. $$
\end{enumerate}
Let now $\bsz=(z_1,z_2)\in [0,1)^2$ and $g(\bst)=\bsone_{\left[\bszero,\bst\right)}(\bsz)$ for $\bst=(t_1,t_2)\in \left[0,1\right)^2$. Let $\bsj\in \NN_{-1}^2$, $\bsm\in \DD_{\bsj}$, $\bsl \in \BB_{\bsj}$ and let $\mu_{\bsj,\bsm,\bsl}(g)$ be the $b$-adic
Haar coefficient of $g$. Then $\mu_{\bsj,\bsm,\bsl}=0$ whenever $\bsz$ is not contained in the interior of the $b$-adic interval $I_{\bsj,\bsm}$. If $\bsz$ is contained in the interior of $I_{\bsj,\bsm}$, then
\begin{enumerate}
   \item If $\bsj=(j_1,j_2)\in \NN_{0}^2$, then there is a $\bsk=(k_1,k_2) \in \{0,1,\dots,b-1\}^2$ such that $\bsz$ is contained in $I_{\bsj,\bsm}^{\bsk}$. Then
	\begin{align*} \mu_{\bsj,\bsm,\bsl}(g)=b^{-j_1-j_2-2}&\left((bm_1+k_1-b^{j_1+1}z_1){\rm e}^{\frac{2\pi\ii}{b}k_1\ell_1}-\sum_{r_1=0}^{k_1-1}{\rm e}^{\frac{2\pi\ii}{b}r_1\ell_1}\right) \times \\ & \times\left((bm_2+k_2-b^{j_2+1}z_2){\rm e}^{\frac{2\pi\ii}{b}k_2\ell_2}-\sum_{r_2=0}^{k_2-1}{\rm e}^{\frac{2\pi\ii}{b}r_2\ell_2}\right). \end{align*}
	\item If $\bsj=(j_1,-1)$ with $j_1 \in \NN_0$, then there is a $k_1\in \{0,1,\dots,b-1\}$ such that $\bsz$ is contained in $I_{\bsj,\bsm}^{(k_1,-1)}$. Then
	  $$ \mu_{\bsj,\bsm,\bsl}(g)=b^{-j_1-1}(1-z_2)\left((bm_1+k_1-b^{j_1+1}z_1){\rm e}^{\frac{2\pi\ii}{b}k_1\ell_1}-\sum_{r_1=0}^{k_1-1}{\rm e}^{\frac{2\pi\ii}{b}r_1\ell_1}\right). $$
		\item If $\bsj=(-1,j_2)$ with $j_2 \in \NN_0$, then there is a $k_2\in \{0,1,\dots,b-1\}$ such that $\bsz$ is contained in $I_{\bsj,\bsm}^{(-1,k_2)}$. Then
	  $$ \mu_{\bsj,\bsm,\bsl}(g)=b^{-j_2-1}(1-z_1)\left((bm_2+k_2-b^{j_2+1}z_2){\rm e}^{\frac{2\pi\ii}{b}k_2\ell_2}-\sum_{r_2=0}^{k_2-1}{\rm e}^{\frac{2\pi\ii}{b}r_2\ell_2}\right). $$
\end{enumerate}
\end{lemma}

\begin{lemma} \label{j1j2} Let $\bsj \in \NN_0^2$ such that $j_1+j_2<n-1$, $\bsm \in \DD_{\bsj}$ and $\bsl \in \BB_{\bsj}$. Then
\begin{align*} 
    \sum_{\bsz \in \cR_{b,n}^{\Sigma}\cap I_{\bsj,\bsm}}&\left((bm_1+k_1-b^{j_1+1}z_1){\rm e}^{\frac{2\pi\ii}{b}k_1\ell_1}-\sum_{r_1=0}^{k_1-1}{\rm e}^{\frac{2\pi\ii}{b}r_1\ell_1}\right) \times \\ & \times\left((bm_2+k_2-b^{j_2+1}z_2){\rm e}^{\frac{2\pi\ii}{b}k_2\ell_2}-\sum_{r_2=0}^{k_2-1}{\rm e}^{\frac{2\pi\ii}{b}r_2\ell_2}\right) \\
		=&\frac{b^{n-j_1-j_2}}{\left({\rm e}^{\frac{2\pi\ii}{b}\ell_1}-1\right)\left({\rm e}^{\frac{2\pi\ii}{b}\ell_2}-1\right)}
		  \pm b^{j_1+j_2-n}\sum_{k_1=0}^{b-1}\sigma^{-1}(k_1){\rm e}^{\frac{2\pi\ii}{b}p(k_1)\ell_1}\sum_{k_2=0}^{b-1}\sigma(k_2){\rm e}^{\frac{2\pi\ii}{b}k_2\ell_2},
\end{align*} where $p(k_1)=k_1$ or $p(k_1)=-k_1-1$ depending on $j_1$ and where the sign depends on $j_2$.
\end{lemma}

\begin{proof}
   With the very same argumentation as in the proof of \cite[Lemma 4.4]{Mar2013} or \cite[Lemma 4.10]{Mar2013b}, we can show that
	\begin{align*} 
    \sum_{\bsz \in \cR_{b,n}^{\Sigma}\cap I_{\bsj,\bsm}}&\left((bm_1+k_1-b^{j_1+1}z_1){\rm e}^{\frac{2\pi\ii}{b}k_1\ell_1}-\sum_{r_1=0}^{k_1-1}{\rm e}^{\frac{2\pi\ii}{b}r_1\ell_1}\right) \times \\ & \times\left((bm_2+k_2-b^{j_2+1}z_2){\rm e}^{\frac{2\pi\ii}{b}k_2\ell_2}-\sum_{r_2=0}^{k_2-1}{\rm e}^{\frac{2\pi\ii}{b}r_2\ell_2}\right) \\
		=&\frac{b^{n-j_1-j_2}}{\left({\rm e}^{\frac{2\pi\ii}{b}\ell_1}-1\right)\left({\rm e}^{\frac{2\pi\ii}{b}\ell_2}-1\right)}\\
		&+\underbrace{b^{j_1+j_2-n}\sum_{k_1=0}^{b-1}a_{n-j_1}{\rm e}^{\frac{2\pi\ii}{b}k_1\ell_1}\sum_{k_2=0}^{b-1}\sigma_{j_2+1}(a_{j_2+1}){\rm e}^{\frac{2\pi\ii}{b}k_2\ell_2}}_{S},
 \end{align*}
where $\sigma_{n-j_1}(a_{n-j_1})=k_1$ and $a_{j_2+1}=k_2$. We analyse the expression $S$. We have
$$   S= b^{j_1+j_2-n}\underbrace{\sum_{k_1=0}^{b-1}\sigma_{n-j_1}^{-1}(k_1){\rm e}^{\frac{2\pi\ii}{b}k_1\ell_1}}_{S_1}\underbrace{\sum_{k_2=0}^{b-1}\sigma_{j_2+1}(k_2){\rm e}^{\frac{2\pi\ii}{b}k_2\ell_2}}_{S_2}.
$$
We have to distinguish the cases $\sigma_{n-j_1}=\sigma$ and $\sigma_{n-j_1}=\overline{\sigma}$ as well as the cases $\sigma_{j_2+1}=\sigma$ and $\sigma_{j_2+1}=\overline{\sigma}$, respectively. The case $\sigma_{n-j_1}=\sigma$ leads to 
$S_1=\sum_{k_1=0}^{b-1}\sigma^{-1}(k_1){\rm e}^{\frac{2\pi\ii}{b}k_1\ell_1}$, whereas $\sigma_{n-j_1}=\overline{\sigma}$ yields
$$S_1=\sum_{k_1=0}^{b-1}\sigma^{-1}(b-1-k_1){\rm e}^{\frac{2\pi\ii}{b}k_1\ell_1}=\sum_{k_1=0}^{b-1}\sigma^{-1}(k_1){\rm e}^{\frac{2\pi\ii}{b}(b-1-k_1)\ell_1}=\sum_{k_1=0}^{b-1}\sigma^{-1}(k_1){\rm e}^{\frac{2\pi\ii}{b}(-1-k_1)\ell_1}.$$
Combining these results, we have $S_1=\sum_{k_1=0}^{b-1}\sigma^{-1}(k_1){\rm e}^{\frac{2\pi\ii}{b}p(k_1)\ell_1}$, where 
$p(k_1)=k_1$ if $\sigma_{n-j_1}=\sigma$ or $p(k_1)=-k_1-1$ if $\sigma_{n-j_1}=\overline{\sigma}$. Hence, $p(k_1)$ depends only on $j_1$. 
The case $\sigma_{j_2+1}=\sigma$ yields $S_2=\sum_{k_2=0}^{b-1}\sigma(k_2){\rm e}^{\frac{2\pi\ii}{b}k_2\ell_2}$, whereas
$\sigma_{j_2+1}=\overline{\sigma}$ leads to $$S_2=\sum_{k_2=0}^{b-1}(b-1-\sigma(k_2)){\rm e}^{\frac{2\pi\ii}{b}k_2\ell_2}=-\sum_{k_2=0}^{b-1}\sigma(k_2){\rm e}^{\frac{2\pi\ii}{b}k_2\ell_2},$$
and therefore we have $S_2=\pm \sum_{k_2=0}^{b-1}\sigma(k_2){\rm e}^{\frac{2\pi\ii}{b}k_2\ell_2}$, where the sign depends only on $j_2$. The proof is complete.
\end{proof}

\begin{lemma} \label{j1} Let $\bsj=(j_1,-1)$ such that $j_1 \in \NN_0$ with $j_1<n-1$, $\bsm=(m_1,0)$ with $m_1 \in \DD_{j_1}$ and $\bsl=(\ell_1,1)$ with $\ell_1 \in \BB_{j_1}$. Then
   \begin{align*} \sum_{\bsz \in \cR_{b,n}^{\Sigma}\cap I_{\bsj,\bsm}}&\left((bm_1+k_1-b^{j_1+1}z_1){\rm e}^{\frac{2\pi\ii}{b}k_1\ell_1}-\sum_{r_1=0}^{k_1-1}{\rm e}^{\frac{2\pi\ii}{b}r_1\ell_1}\right)(1-z_2) \\ =& \frac{b^{n-j_1}(1-2\varepsilon)+b}{2\left({\rm e}^{\frac{2\pi\ii}{b}\ell_1}-1\right)}+\frac{b^{-1}-b^{j_1-n}}{2}\sum_{k_1=0}^{b-1}\sigma^{-1}(k_1){\rm e}^{\frac{2\pi\ii}{b}p(k_1)\ell_1} \\
	&+\frac{b^{-1}}{{\rm e}^{\frac{2\pi\ii}{b}\ell_1}-1}\left(\sum_{k_1=0}^{b-1}\sigma^{-1}(k_1){\rm e}^{\frac{2\pi\ii}{b}p(k_1)\ell_1}-\frac{b(b-1)}{2}\right),
	\end{align*}
	where $\varepsilon$ is a positive real number depending on $j_1$ and $m_1$ which satisfies $\varepsilon b^{n-j_1} \leq b$ and where $p(k_1)=k_1$ or $p(k_1)=-k_1-1$ depending on $j_1$.
\end{lemma}

 \begin{proof}
 With the very same argumentation as in the proof of \cite[Lemma 4.10]{Mar2013} or \cite[Lemma 4.17]{Mar2013b}, we can show that
\begin{align*} \sum_{\bsz \in \cR_{b,n}^{\Sigma}\cap I_{\bsj,\bsm}}&\left((bm_1+k_1-b^{j_1+1}z_1){\rm e}^{\frac{2\pi\ii}{b}k_1\ell_1}-\sum_{r_1=0}^{k_1-1}{\rm e}^{\frac{2\pi\ii}{b}r_1\ell_1}\right)(1-z_2) \\
	 =& \frac{b^{n-j_1}(1-2\varepsilon)+b}{2\left({\rm e}^{\frac{2\pi\ii}{b}\ell_1}-1\right)}+\underbrace{\sum_{h=1}^{b^{n-j_1-1}}hb^{j_1-n+1}b^{j_1-n}\sum_{k_1=0}^{b-1}a_{n-j_1}{\rm e}^{\frac{2\pi\ii}{b}k_1\ell_1}}_{T_1} \\
	&+ \underbrace{\sum_{h=0}^{b^{n-j_1-1}-1}b^{j_1-n}\sum_{k_1=0}^{b-1}a_{n-j_1}\sum_{r_1=0}^{k_1-1}{\rm e}^{\frac{2\pi\ii}{b}r_1\ell_1}}_{T_2},
\end{align*}
where $\sigma_{n-j_1}(a_{n-j_1})=k_1$. Analogously as in the proof of Lemma~\ref{j1j2}, we find
$$T_1=\frac{b^{-1}-b^{j_1-n}}{2}\sum_{k_1=0}^{b-1}\sigma^{-1}(k_1){\rm e}^{\frac{2\pi\ii}{b}p(k_1)\ell_1},$$
where the value of $p(k_1)$ depends only on $j_1$. We also obtain
\begin{align*}
   T_2=& \frac{1}{{\rm e}^{\frac{2\pi\ii}{b}\ell_1}-1}b^{n-j_1-1}b^{j_1-n}\sum_{k_1=0}^{b-1}a_{n-j_1}\left({\rm e}^{\frac{2\pi\ii}{b}k_1\ell_1}-1\right) \\
	  =& \frac{b^{-1}}{{\rm e}^{\frac{2\pi\ii}{b}\ell_1}-1}\left(\sum_{k_1=0}^{b-1}a_{n-j_1}{\rm e}^{\frac{2\pi\ii}{b}k_1\ell_1}-\frac{b(b-1)}{2}\right) \\
		=& \frac{b^{-1}}{{\rm e}^{\frac{2\pi\ii}{b}\ell_1}-1}\left(\sum_{k_1=0}^{b-1}\sigma^{-1}(k_1){\rm e}^{\frac{2\pi\ii}{b}p(k_1)\ell_1}-\frac{b(b-1)}{2}\right).
\end{align*}
The proof is complete.
\end{proof}

\begin{lemma} \label{j2} Let $\bsj=(-1,j_2)$ such that $j_2 \in \NN_0$ with $j_2<n-1$, $\bsm=(0,m_2)$ with $m_2 \in \DD_{j_2}$ and $\bsl=(1,\ell_2)$ with $\ell_2 \in \BB_{j_2}$. Then
   \begin{align*} \sum_{\bsz \in \cR_{b,n}^{\Sigma}\cap I_{\bsj,\bsm}}&(1-z_1)\left((bm_2+k_2-b^{j_2+1}z_2){\rm e}^{\frac{2\pi\ii}{b}k_2\ell_2}-\sum_{r_2=0}^{k_2-1}{\rm e}^{\frac{2\pi\ii}{b}r_2\ell_2}\right) \\ =& \frac{b^{n-j_2}(1-2\varepsilon)+b}{2\left({\rm e}^{\frac{2\pi\ii}{b}\ell_2}-1\right)}\pm\frac{b^{-1}-b^{j_2-n}}{2}\sum_{k_2=0}^{b-1}\sigma(k_2){\rm e}^{\frac{2\pi\ii}{b}k_2\ell_2} \\
	&+\frac{b^{-1}}{{\rm e}^{\frac{2\pi\ii}{b}\ell_2}-1}\left(\pm\sum_{k_2=0}^{b-1}\sigma(k_2){\rm e}^{\frac{2\pi\ii}{b}k_2\ell_2}-\frac{b(b-1)}{2}\right),
	\end{align*}
		where $\varepsilon'$ is a positive real number depending on $j_2$ and $m_2$ which satisfies $\varepsilon' b^{n-j_2}\leq b$ and where the signs depend only on $j_2$.
\end{lemma}

\begin{proof}
   This fact follows from
	\begin{align*} \sum_{\bsz \in \cR_{b,n}^{\Sigma}\cap I_{\bsj,\bsm}}& (1-z_1)\left((bm_2+k_2-b^{j_2+1}z_2){\rm e}^{\frac{2\pi\ii}{b}k_2\ell_2}-\sum_{r_2=0}^{k_2-1}{\rm e}^{\frac{2\pi\ii}{b}r_2\ell_2}\right) \\
	 =& \frac{b^{n-j_2}(1-2\varepsilon')+b}{2\left({\rm e}^{\frac{2\pi\ii}{b}\ell_2}-1\right)}+\sum_{h=1}^{b^{n-j_2-1}}hb^{j_2-n+1}b^{j_2-n}\sum_{k_2=0}^{b-1}a_{j_2+1}{\rm e}^{\frac{2\pi\ii}{b}k_2\ell_2} \\
	&+\sum_{h=0}^{b^{n-j_2-1}-1}b^{j_2-n}\sum_{k_2=0}^{b-1}\sigma_{j_2+1}(a_{j_2+1})\sum_{r_2=0}^{k_2-1}{\rm e}^{\frac{2\pi\ii}{b}k_2\ell_2}
\end{align*} and the relation $a_{j_2+1}=k_2$. The argumentation is very similar to the proofs of \cite[Lemma 4.10]{Mar2013}, \cite[Lemma 4.17]{Mar2013b},  Lemma~\ref{j1j2} and Lemma~\ref{j1}.
\end{proof}

\begin{lemma}\label{summary} Let $\bsj\in \NN_{-1}^2$, $\bsm \in \DD_{\bsj}$, $\bsl \in \BB_{\bsj}$ and $\mu_{\bsj,\bsm,\bsl}$ be the $b$-adic Haar coefficients of the local discrepancy of $\cR_{b,n}^{\Sigma}$. We recall the definition $|\bsj|=\max\{0,j_1\}+\max\{0,j_2\}$. Then
\begin{enumerate}
   \item if $\bsj \in \NN_0^2$ and $|\bsj|<n-1$, then
	   $$ |\mu_{\bsj,\bsm,\bsl}|\leq \left(\frac{b-1}{2}\right)^2b^{-2n}, $$
		\item if $\bsj \in \NN_0^2$, $|\bsj|\geq n-1$ and $j_1,j_2 \leq n$, then $|\mu_{\bsj,\bsm,\bsl}|\leq cb^{-n-|\bsj|}$ for some constant $c>0$ and
		$$ |\mu_{\bsj,\bsm,\bsl}|=\frac{b^{-2|\bsj|-2}}{\left|{\rm e}^{\frac{2\pi\ii}{b}\ell_1}-1\right|\left|{\rm e}^{\frac{2\pi\ii}{b}\ell_2}-1\right|} $$
		for all but $b^n$ coefficients $\mu_{\bsj,\bsm,\bsl}$,
		\item if $\bsj \in \NN_0^2$ and $j_1\geq n$ or $j_2\geq n$, then
		$$ |\mu_{\bsj,\bsm,\bsl}|= \frac{b^{-2|\bsj|-2}}{\left|{\rm e}^{\frac{2\pi\ii}{b}\ell_1}-1\right|\left|{\rm e}^{\frac{2\pi\ii}{b}\ell_2}-1\right|},$$
		\item if $\bsj=(j_1,-1)$ or $\bsj=(-1,j_2)$ with $j_1\in \NN_0$, $j_1<n$ or $j_1\in \NN_0$, $j_2<n$          respectively, then we have $$ |\mu_{\bsj,\bsm,\bsl}|\leq (b^2-1)b^{-n-j_i} $$
		for $i=1$ and $i=2$, respectively,
		\item if $\bsj=(j_1,-1)$ or $\bsj=(-1,j_2)$ with $j_1\in \NN_0$, $j_1\geq n$ or $j_2\geq n$          respectively, then we have $$ |\mu_{\bsj,\bsm,\bsl}|=\frac{1}{2}\frac{b^{-2j_i-1}}{\left|{\rm e}^{\frac{2\pi\ii}{b}\ell_i}-1\right|} $$ for $i=1$ and $i=2$, respectively.
\end{enumerate}
\end{lemma}
\begin{proof}
   Point (2) can be verified analogously as \cite[Proposition 5.1, (ii)]{Mar2013} or \cite[Proposition 4.18, (ii)]{Mar2013b}. Point (3) and Point (5) follow
	from Lemma~\ref{linearpart} and the fact that there are no points contained in the interior of $I_{\bsj,\bsm}$ for $\bsj \in \NN_{-1}^2$ if $j_1\geq n$ or $j_2 \geq n$. For the verification of Point (1) we use Lemma~\ref{linearpart} and Lemma~\ref{j1j2} and obtain
	\begin{align*} 
	    \mu_{\bsj,\bsm,\bsl}=&\frac{1}{b^n}b^{-j_1-j_2-2}\left(\frac{b^{n-j_1-j_2}}{\left({\rm e}^{\frac{2\pi\ii}{b}\ell_1}-1\right)\left({\rm e}^{\frac{2\pi\ii}{b}\ell_2}-1\right)} \right. \\
		  & \left. \pm b^{j_1+j_2-n}\sum_{k_1=0}^{b-1}\sigma^{-1}(k_1){\rm e}^{\frac{2\pi\ii}{b}p(k_1)\ell_1}\sum_{k_2=0}^{b-1}\sigma(k_2){\rm e}^{\frac{2\pi\ii}{b}k_2\ell_2}\right)-\frac{b^{-2j_1-2j_2-2}}{\left({\rm e}^{\frac{2\pi\ii}{b}\ell_1}-1\right)\left({\rm e}^{\frac{2\pi\ii}{b}\ell_2}-1\right)} \\
			=& \pm b^{-2n-2}\sum_{k_1=0}^{b-1}\sigma^{-1}(k_1){\rm e}^{\frac{2\pi\ii}{b}p(k_1)\ell_1}\sum_{k_2=0}^{b-1}\sigma(k_2){\rm e}^{\frac{2\pi\ii}{b}k_2\ell_2},
	\end{align*}
	which leads to
	\begin{align*}
	  |\mu_{\bsj,\bsm,\bsl}|=& b^{-2n-2}\left|\sum_{k_1=0}^{b-1}\sigma^{-1}(k_1){\rm e}^{\frac{2\pi\ii}{b}p(k_1)\ell_1}\right|\left|\sum_{k_2=0}^{b-1}\sigma(k_2){\rm e}^{\frac{2\pi\ii}{b}k_2\ell_2}\right| \\
		\leq &b^{-2n-2}\sum_{k_1=0}^{b-1}\sigma^{-1}(k_1)\sum_{k_2=0}^{b-1}\sigma(k_2)=b^{-2n-2}\left(\frac{b(b-1)}{2}\right)^2=\left(\frac{b-1}{2}\right)^2b^{-2n}
	\end{align*}
	as claimed, since with $k$ also $\sigma^{-1}(k)$ and $\sigma(k)$ runs through $\{0,1,\dots,b-1\}$, respectively. We turn to the case that $\bsj=(j_1,-1)$ with $j_1\in \NN_0$, $j_1<n$ and therefore regard Lemma~\ref{linearpart} and Lemma~\ref{j1}.
	We have
	\begin{align*}
	   \mu_{\bsj,\bsm,\bsl}=& \frac{1}{b^n}b^{-j_1-1}\left(\frac{b^{n-j_1}(1-2\varepsilon)+b}{2\left({\rm e}^{\frac{2\pi\ii}{b}\ell_1}-1\right)}+\frac{b^{-1}-b^{j_1-n}}{2}\sum_{k_1=0}^{b-1}\sigma^{-1}(k_1){\rm e}^{\frac{2\pi\ii}{b}p(k_1)\ell_1} \right. \\
	& \left. +\frac{b^{-1}}{{\rm e}^{\frac{2\pi\ii}{b}\ell_1}-1}\left(\sum_{k_1=0}^{b-1}\sigma^{-1}(k_1){\rm e}^{\frac{2\pi\ii}{b}p(k_1)\ell_1}-\frac{b(b-1)}{2}\right)\right)-\frac{b^{-2j_1-1}}{2\left({\rm e}^{\frac{2\pi\ii}{b}\ell_1}-1\right)} \\
	=& -\frac{b^{-2j_1-1}\varepsilon}{{\rm e}^{\frac{2\pi\ii}{b}\ell_1}-1}+\frac{b^{-j_1-n}}{2\left({\rm e}^{\frac{2\pi\ii}{b}\ell_1}-1\right)}+\frac{b^{-j_1-n-2}-b^{-2n-1}}{2}\sum_{k_1=0}^{b-1}\sigma^{-1}(k_1){\rm e}^{\frac{2\pi\ii}{b}p(k_1)\ell_1} \\ &+\frac{b^{-j_1-n-2}}{{\rm e}^{\frac{2\pi\ii}{b}\ell_1}-1}\left(\sum_{k_1=0}^{b-1}\sigma^{-1}(k_1){\rm e}^{\frac{2\pi\ii}{b}p(k_1)\ell_1}-\frac{b(b-1)}{2}\right).
	\end{align*}
	The triangle inequality yields (since $\varepsilon b^{n-j_1}\leq b$ and $b^{-2n-1}\leq b^{n-j_1-2}$)
	\begin{align*}  |\mu_{\bsj,\bsm,\bsl}|\leq&\frac{b^{-2j_1-1}\varepsilon}{|{\rm e}^{\frac{2\pi\ii}{b}\ell_1}-1|}+\frac{b^{-j_1-n}}{2|{\rm e}^{\frac{2\pi\ii}{b}\ell_1}-1|}+\frac{b^{-j_1-n-2}+b^{-2n-1}}{2}\sum_{k_1=0}^{b-1}\sigma^{-1}(k_1) \\ &+\frac{b^{-j_1-n-2}}{|{\rm e}^{\frac{2\pi\ii}{b}\ell_1}-1|}\left(\sum_{k_1=0}^{b-1}\sigma^{-1}(k_1)+\frac{b(b-1)}{2}\right) \\
		\leq & \frac{b^{-j_1-n}}{|{\rm e}^{\frac{2\pi\ii}{b}\ell_1}-1|}+\frac{b^{-j_1-n}}{2|{\rm e}^{\frac{2\pi\ii}{b}\ell_1}-1|}+\frac{b^{-j_1-n-2}+b^{-2n-1}}{2}\frac{b(b-1)}{2}+\frac{b^{-j_1-n-2}}{|{\rm e}^{\frac{2\pi\ii}{b}\ell_1}-1|}b(b-1) \\
		\leq &\frac{5}{2}\frac{b^{-j_1-n}}{|{\rm e}^{\frac{2\pi\ii}{b}\ell_1}-1|}+\frac{b^{-j_1-n}}{2}
		  \leq \left(\frac{5}{2}\frac{b^2-1}{6}+\frac{1}{2}\right)b^{-j_1-n}\leq (b^2-1)b^{-j_1-n},
	\end{align*}
	where we used Lemma~\ref{exponential}.
	The case $(-1,j_2)$ can be handled completely analogously.
\end{proof}

For the proof of Theorem~\ref{theosym} we also need upper bounds on the absolute values of the Haar coefficients $\mu_{\bsj,\bsm,\bsl}^{\Sigma, {\rm sym}}= \langle  D_{\widetilde{N}}({\cal R}_{n,b}^{\Sigma,{\rm sym}}, \, \cdot \,), h_{\bsj,\bsm,\bsl} \rangle$ which are given in the following:

\begin{lemma}  \label{HCSYM} Let $\bsj=(j_1,j_2)\in \NN_{-1}^2$. Then in the case $\bsj\neq (-1,-1)$ we have $$|\mu_{\bsj,\bsm,\bsl}^{\Sigma,{\rm sym}}| \le |\mu_{\bsj,\bsm,\bsl}|\ \ \ \mbox{ for all } \ \bsm \in \mathbb{D}_{\bsj}, \bsl \in \BB_{\bsj},$$ where the coefficients $\mu_{\bsj,\bsm,\bsl}$ refer to $D_N(\cR_{b,n}^{\Sigma},\cdot)$. Hence the results in Lemma~\ref{summary} apply accordingly also to $|\mu_{\bsj,\bsm,\bsl}^{\Sigma,{\rm sym}}|$. In the case $\bsj= (-1,-1)$
we have $$\mu_{(-1,-1),(0,0),(1,1)}^{\Sigma,{\rm sym}} = \frac{1}{\widetilde{N}}+\frac{1}{{\widetilde{N}}^2}.$$ 
(Recall that $\widetilde{N}=2b^n$ is the number of points in ${\cal R}_{n,b}^{\Sigma,{\rm sym}}$.)
\end{lemma}

\begin{proof} Analogously as in the proof of \cite[Lemma 3]{HKP14}, we obtain \[ \mu_{\bsj,\bsm,\bsl}^{\Sigma,{\rm sym}}=\frac{1}{2}\left(\mu_{\bsj,\bsm,\bsl}^{\Sigma}+\mu_{\bsj,\bsm,\bsl}^{\Sigma^{\ast}}\right),\] where here we write $\mu_{\bsj,\bsm,\bsl}^{\Sigma}$ for the Haar coefficients of the local discrepancy of ${\cR}_{n,b}^{\Sigma}$ in order to stress the dependence on $\Sigma$ and accordingly for $\mu_{\bsj,\bsm,\bsl}^{\Sigma^{\ast}}$. This relation together with Lemma~\ref{haar1} immediately leads to the assertion in the case $\bsj=(-1,-1)$ (simply observe that $\Sigma^{\ast}$ contains $n-l_n$ components equal to $\sigma$ whenever $\Sigma$ has $l_n$ such components), whereas in the case $\bsj \neq (-1,-1)$ we apply the triangle inequality to obtain \[ |\mu_{\bsj,\bsm,\bsl}^{\Sigma,{\rm sym}}|\leq \frac{1}{2}\left(|\mu_{\bsj,\bsm,\bsl}^{\Sigma}|+|\mu_{\bsj,\bsm,\bsl}^{\Sigma^{\ast}}|\right).\]
Since the bounds on $|\mu_{\bsj,\bsm,\bsl}^{\Sigma}|$ given in Lemma~\ref{summary} do not depend on $\Sigma$, this inequality completes the proof.
\end{proof}

\begin{remark} \rm Since $\mu_{(-1,-1),(0,0),(1,1)}(D_N(\cR_{b,n}^{\Sigma},\cdot))$ does not depend on the order of the components in $\Sigma$, but only on the number of $\sigma$-entries, Lemma~\ref{HCSYM} is true for every point set of the form
$ \cR_{b,n}^{\Sigma_1} \cup \cR_{b,n}^{\Sigma_2} $, where $\Sigma_1,\Sigma_2\in\{\sigma,\overline{\sigma}\}^n$ and where $\Sigma_1$ has $l_n$ entries equal to $\sigma$ and $\Sigma_2$ has $n-l_n$
of such entries for any $l_n\in \{0,1,\dots,n\}$. The proof of Theorem~\ref{theosym} therefore works also for point sets of this kind. However, these point sets are in general not symmetrized in the sense of \eqref{warumsym}.
\end{remark}

\section{Proofs of Theorem~\ref{theovdc}, Theorem~\ref{theohamm} and Theorem~\ref{theosym}}\label{secproofthm}

\begin{proof2} From Proposition~\ref{equivalence} it follows that it suffices to show
  $$  \left(\sum_{j \in \NN_{-1}}b^{j\left(r-\frac{1}{p}+1\right)q}\left( \sum_{m \in \DD_{j},\ell\in \BB_{j}}\left| \mu_{j,m,\ell}^{N,\sym}\right|^p\right)^{\frac{q}{p}}\right)^{\frac{1}{q}} \ll
	  \begin{cases} N^{-1}(\log{N})^{\frac{1}{q}} & \mbox{ if } r=0, \\
		               N^{r-1} & \mbox{ if } 0<r<\frac{1}{p}. \end{cases} $$
Since $q\geq 1,
$, we have
\begin{align*}
   &\left(\sum_{j \in \NN_{-1}}b^{j\left(r-\frac{1}{p}+1\right)q}\left( \sum_{m \in \DD_{j},\ell\in \BB_{j}}\left| \mu_{j,m,\ell}^{N,\sym}\right|^p\right)^{\frac{q}{p}}\right)^{\frac{1}{q}} \ll \left|\mu_{-1,0,1}\right| \\
   &+\left(\sum_{j=0}^{\lceil \log_{b}{N} \rceil-1}b^{j\left(r-\frac{1}{p}+1\right)q}\left( \sum_{m \in \DD_{j},\ell\in \BB_{j}}\left| \mu_{j,m,\ell}^{N,\sym}\right|^p\right)^{\frac{q}{p}}\right)^{\frac{1}{q}}\\ 
   &+\left(\sum_{j=\lceil \log_{b}{N} \rceil}^{\infty}b^{j\left(r-\frac{1}{p}+1\right)q}\left( \sum_{m \in \DD_{j},\ell\in \BB_{j}}\left| \mu_{j,m,\ell}^{N,\sym}\right|^p\right)^{\frac{q}{p}}\right)^{\frac{1}{q}}=:S_1+S_2+S_3.
\end{align*}
We apply Lemma~\ref{erster} and Corollary~\ref{coro1}. We have $S_1 \ll N^{-1}\ll N^{r-1}$ for all $0\leq r < \frac{1}{p}$. We also find
\begin{align*}
   S_2 &\ll \left(\sum_{j=0}^{\lceil \log_{b}{N} \rceil-1}b^{j\left(r-\frac{1}{p}+1\right)q}\left(  b^j\left(\frac{1}{N}\frac{1}{b^j}\right)^p\right)^{\frac{q}{p}}\right)^{\frac{1}{q}} =\frac{1}{N}\left(\sum_{j=0}^{\lceil \log_{b}{N} \rceil-1}b^{jqr}\right)^{\frac{1}{q}}.
\end{align*}
The assumption $r=0$ leads to
$$S_2 \ll \frac{1}{N}\left(\sum_{j=0}^{\lceil \log_{b}{N} \rceil-1}1\right)^{\frac{1}{q}}\ll N^{-1}(\log{N})^{\frac{1}{q}},$$
whereas for $0< r < \frac{1}{p}$ we obtain
$$S_2 \ll \frac{1}{N}\left(\sum_{j=0}^{\lceil \log_{b}{N} \rceil-1}b^{jqr}\right)^{\frac{1}{q}}\ll \frac{1}{N}(b^{\log_{b}{N}})^r=N^{r-1}.$$
It remains to estimate $S_3$. We have
\begin{align*}S_3 &\ll \left(\sum_{j=\lceil \log_{b}{N} \rceil}^{\infty}b^{j\left(r-\frac{1}{p}+1\right)q}\left(  b^j\left(\frac{1}{b^{2j}}\right)^p\right)^{\frac{q}{p}}\right)^{\frac{1}{q}}=\left(\sum_{j=\lceil \log_{b}{N} \rceil}^{\infty}b^{jq(r-1)}\right)^{\frac{1}{q}} \\
&\ll b^{\log_{b}{N}(r-1)}=N^{r-1},
\end{align*}
which concludes the proof of Theorem~\ref{theovdc}.
\end{proof2} 

\begin{proof3} The bounds on the Haar coefficients of the digit scrambled Hammersley point set we found in Lemma~\ref{summary} are of the same order of magnitude in $N$ as the bounds given in \cite[Proposition 5.1]{Mar2013}, except for the coefficient $\mu_{(-1,-1),(0,0),(1,1)}$. However, provided that $\Sigma$ is such that $|2l_n-n|=O(n^{\frac{1}{q}})$ or $\frac{1}{b}\sum_{a=0}^{b-1}\sigma(a)a=\frac{(b-1)^2}{4}$, this Haar coefficient is of order $N^{-1}(\log{N})^{\frac{1}{q}}$ (see Remark~\ref{perm}). This order is small enough to apply the very same method as used in the proof of \cite[Theorem 1.1]{Mar2013} to show the sufficiency of the condition in Theorem~\ref{theohamm}.
The necessity of the condition given on $l_n$ or $\sigma$ follows from Proposition~\ref{equivalence}, which
gives \begin{align*} \left\|D_N(\cR_{b,n}^{\Sigma},\cdot) \mid S_{p,q}^rB([0,1)^2)\right\| \gg &  \left(\sum_{\bsj \in \NN_{-1}^2}b^{(j_1+j_2)\left(r-\frac{1}{p}+1\right)q}\left( \sum_{\bsm \in \DD_{j},\bsl\in \BB_{\bsj}}\left| \mu_{\bsj,\bsm,\bsl}\right|^p\right)^{\frac{q}{p}}\right)^{\frac{1}{q}} \\ \gg & |\mu_{(-1,-1),(0,0),(1,1)}|. \end{align*}
From this and Lemma~\ref{haar1} it is evident that we must have $|2l_n-n|=O(n^{\frac{1}{q}})$ or $\frac{1}{b}\sum_{a=0}^{b-1}\sigma(a)a=\frac{(b-1)^2}{4}$ in order to reach the optimal $L_p$ discrepancy bound.
\end{proof3}

\begin{proof4} The proof is obvious, since the bounds on the Haar coefficients of $ D_{\widetilde{N}}({\cR}_{n,b}^{\Sigma,{\rm sym}}, \, \cdot \,)$ are the same as for $ D_N({\cR}_{n,b}^{\Sigma}, \, \cdot \,)$, except for the coefficient $\mu_{(-1,-1),(0,0),(1,1)}$, which is of order $\frac{1}{\widetilde{N}}$ independently of $\Sigma$. So we can simply refer to the proof of \cite[Theorem 1.1]{Mar2013} again.
\end{proof4}

\section*{Acknowledgements}{\small
The author would like to thank Friedrich Pillichshammer for his corrections and suggestions to improve the presentation.}

\noindent{\bf Author's Address:}

\noindent Ralph Kritzinger, Institut f\"{u}r Finanzmathematik und angewandte Zahlentheorie, Johannes Kepler Universit\"{a}t Linz, Altenbergerstra{\ss}e 69, A-4040 Linz, Austria. Email: ralph.kritzinger(at)jku.at


\begin{thebibliography}{10} \addcontentsline{toc}{chapter}{Bibliography}
\bibitem{bilyk} D. Bilyk, M. Lacey, I. Parissis, A. Vagharshakyan, Exponential squared integrability of the discrepancy function in two dimensions, Mathematika 55 (2009) 1--27.

\bibitem{chen1980} W.W.L. Chen, On irregularities of distribution, Mathematika 27 (1980) 153--170. 


\bibitem{CS02} W.W.L. Chen, M.M. Skriganov, Explicit constructions in the classical mean squares problem in irregularities of point distribution,  J. Reine Angew. Math. 545 (2002) 67--95.

\bibitem{CS03} W.W.L. Chen, M.M. Skriganov, Davenport's theorem in the theory of irregularities of point distribution,  J. Math. Sci. 545 (2003) 2076--2084.

\bibitem{daven} H. Davenport, Note on irregularities of distribution, Mathematika 3 (1956) 131--135.


\bibitem{DHP} J. Dick, A. Hinrichs and F. Pillichshammer, Proof techniques in quasi-Monte Carlo theory, J. Complexity 31 (2015) 327--371.

\bibitem{DP10} J. Dick, F. Pillichshammer, Digital nets and sequences, in: Discrepancy theory and quasi-Monte Carlo integration, Cambridge University Press, Cambridge, 2010.

\bibitem{DP14neu} J. Dick, F. Pillichshammer, Optimal $\mathcal{L}_2$ discrepancy bounds for higher order digital sequences over the finite field $\mathbb{F}_2$, Acta Arith. 162 (2014) 65--99.

\bibitem{DP14} J. Dick, F. Pillichshammer, Explicit constructions of point sets and sequences with low discrepancy, in: P. Kritzer, H. Niederreiter, F. Pillichshammer and A. Winterhof (Eds.), Uniform Distribution and Quasi-Monte Carlo Methods. Discrepancy, Integration and Applications, De Gruyter, Berlin, 2014, pp. 63--86.

\bibitem{DT97} M. Drmota, R.F. Tichy, Sequences, discrepancies and applications. Lecture Notes in Mathematics 1651, Springer Verlag, Berlin, 1997.


\bibitem{fau81} H. Faure, Discr\'{e}pance de suites associ\'{e}es \`{a} un syst\`{e}me de num\'{e}ration (en dimension un), Bull. Soc. Math. France 109 (1981) 143--182.


\bibitem{FP09} H. Faure, F. Pillichshammer, $L_2$ discrepancy of two-dimensional digitally shifted Hammersley point sets in base $b$, in: P. L'Ecuyer and A. Owen (Eds.), Monte Carlo and Quasi-Monte Carlo Methods 2008, Springer Verlag, 2009, pp. 355--368

\bibitem{FPPS09} H. Faure, F. Pillichshammer, G. Pirsic, and W. Ch. Schmid, $L_2$-discrepancy of generalized two-dimensional Hammersley point
  sets scrambled with arbitrary permutations, Acta Arith. 141 (2010) 395--418.


\bibitem{hala} G. Hal\'{a}sz, On Roth's method in the theory of irregularities of point distributions, in: Recent progress in analytic number theory, Vol. 2, Academic Press, London-New York, 1981, pp. 79--94. 

\bibitem{HZ} J.H. Halton, S.K. Zaremba, The extreme and $L^2$ discrepancies of some plane sets, Monatsh. Math. 73 (1969) 316--328.

\bibitem{hin2010} A. Hinrichs, Discrepancy of Hammersley points in Besov spaces of dominating mixed smoothness, Math. Nachr. 283 (2010) 478--488.

\bibitem{HKP14} A. Hinrichs, R. Kritzinger, F. Pillichshammer, Optimal order of $L_p$-discrepancy of digit shifted Hammersley point sets in dimension 2,  Unif. Distrib. Theory 10 (2015) 115--133.

\bibitem{KP2015} R. Kritzinger, F. Pillichshammer, $L_p$-discrepancy of the symmetrized van der Corput sequence, Arch. Math. 104 (2015) 407--418.

\bibitem{lp2001} G. Larcher, F. Pillichshammer, Walsh series
analysis of the $L_2$ discrepancy of symmetrisized point sets,
Monatsh. Math. 132 (2001) 1--18.



\bibitem{LP14} G. Leobacher, F. Pillichshammer, Introduction to quasi-Monte Carlo integration and applications, Compact Textbooks in Mathematics, Birkh\"auser, 2014.

\bibitem{kuinie} L. Kuipers, H. Niederreiter, Uniform distribution of sequences, John Wiley, New York, 1974.

\bibitem{Mar2013} L. Markhasin, Discrepancy of generalized Hammersley type point sets in Besov spaces with dominating mixed smoothness, Unif. Distrib. Theory 8 (2013) 135--164. 

\bibitem{Mar2013a}  L. Markhasin, Quasi-Monte Carlo methods for integration of functions with dominating mixed smoothness in arbitrary dimension. J. Complexity 29 (2014) 370--388.

\bibitem{Mar2013b}  L. Markhasin, Discrepancy and integration in function spaces with dominating mixed smoothness, Dissertationes Mathematicae 494 (2013) 1--81.

\bibitem{mat} J. Matou\v{s}ek, Geometric discrepancy. An illustrated guide, Algorithms and Combinatorics 18, Springer-Verlag, Berlin, 1999.


\bibitem{Nied92} H. Niederreiter, Random number generation and quasi-Monte Carlo methods. Number 63 in CBMS-NFS Series in Applied Mathematics, SIAM, Philadelphia, 1992.

\bibitem{pro86} P.D. Proinov, On irregularities of distribution,  C. R. Acad. Bulgare Sci. 39 (1986) 31--34.



\bibitem{Roth} K.F. Roth, On irregularities of distribution, Mathematika 1 (1954) 73--79.


\bibitem{schX} W.M. Schmidt, Irregularities of distribution. X, in: Number Theory and Algebra, Academic Press, New York, 1977, pp. 311--329.

\bibitem{Skr} M.M. Skriganov, Harmonic analysis on totally disconnected groups and irregularities of point distributions, J. Reine Angew. Math. 600 (2006) 25--49.



\bibitem{Tri10} H. Triebel, Bases in function spaces, sampling, discrepancy, numerical integration, European Math. Soc. Publishing House, Z\"urich, 2010.

\bibitem{Tri10b} H. Triebel, Numerical integration and discrepancy, a new approach, Math. Nachr. 283 (2010) 139--159.

\end{thebibliography}
\end{document}